\documentclass[12pt,draftcls,onecolumn]{IEEEtran}
\IEEEoverridecommandlockouts
\usepackage{relsize}
\usepackage{amsmath,bbm}
\usepackage{amsfonts}
\usepackage{amsmath, amsthm, amssymb}
\usepackage{tikz}
\usepackage{pgf, hyperref, cite}
\usepackage{dsfont}
\usepackage{subfig}
\usepackage{relsize}
\usetikzlibrary{arrows,automata}
\usepackage{algorithm}
\usepackage[noend]{algpseudocode}
\usepackage{comment}
\usepackage{graphicx}
\usepackage{subfig}

\newtheorem{remark}{Remark}
\newtheorem{assumption}{Assumption}
\newtheorem{definition}{Definition}

\newtheorem{properties}{Properties}
\newcommand{\be}{\begin{equation}}
\newcommand{\ee}{\end{equation}}
\newcommand{\req}[1]{(\ref{#1})}

\usepackage{setspace}

\def\ew#1{{{\color{black}#1}}}
\def\fm#1{{{\color{black}#1}}}
\def\sm#1{{{\color{black}#1}}}
\def\sf#1{{{\color{black}#1}}}
\def\ss#1{{{\color{black}#1}}}
\def\re#1{{{\color{black}#1}}}
\usepackage{mathtools}

\DeclarePairedDelimiter\floor{\lfloor}{\rfloor}
\newtheorem{theorem}{Theorem}[section]
\newtheorem{lemma}[theorem]{Lemma}

\newcommand{\norm}[1]{\left|\left|#1\right|\right|}

\providecommand{\boldsymbol}[1]{\mbox{\boldmath $#1$}}

\title{A Fast Distributed Asynchronous Newton-Based Optimization Algorithm $^*$
\thanks{$^*$ Support provided by Leslie and Mac McQuown and DARPA Lagrange} }
\author{Fatemeh Mansoori$^\dag$\thanks{$^\dag$Department of Electrical Engineering and Computer Science, Northwestern University, Email: fatemehmansoori2019@u.northwestern.edu} and Ermin Wei$^\dag$}
\date{}

\begin{document}
\maketitle
\begin{abstract}
One of the most important problems in the field of distributed optimization is the problem of minimizing a sum of local convex objective functions over a networked system. Most of the existing \re{work in this area focus} on developing distributed algorithms in a synchronous setting under the presence of a central clock, where the agents need to wait for the slowest one to finish the update, before proceeding to the next iterate. Asynchronous distributed algorithms remove the need for a central coordinator, reduce the synchronization wait, and allow some agents to compute faster and execute more iterations. In the asynchronous setting, the only known algorithms for solving this problem could achieve either linear or sublinear rate of convergence. In this work, we built upon the existing literature to develop and analyze an asynchronous Newton-based method to solve a penalized version of the problem. We show that this algorithm guarantees almost sure convergence with global linear and local quadratic rate in expectation. Numerical studies confirm superior performance of our algorithm against other asynchronous methods. 
\end{abstract}
\begin{IEEEkeywords}
Optimization algorithms,  Asynchronous algorithms, Network analysis and control, Agents and autonomous systems.
\end{IEEEkeywords}
\section{Introduction}\label{sec:intro}
Along with the advancement of the modern technology, the complexity and size of the problems and datasets are growing rapidly in different areas such as machine learning, signal processing, and sensor networks. As a result, the datasets are too large to be processed on a single processor or they might be collected or stored in a distributed manner. Therefore,  centralized access to the information is not possible and it is crucial to deploy distributed control and optimization algorithms, which rely only on local information, processing, and communication. Distributed optimization algorithms are implemented over a network of connected agents (or processors) , where each agent solves a smaller subproblem \cite{bo11, bot16,  aliLinMorse}, \re{\cite{koppel2015saddle}},\cite{no4,wei2013distributed}. \par
\ss{A fundamental problem requiring distributed optimization is the  problem of minimizing a sum of local objective functions, i.e., $\min_x\sum_{i=1}^nf_i(x)$, where each agent $i$ in the network has access to a component of the objective function, $f_i$. Such a problem can be solved in a distributed way by defining local copies of the decision variable for the agents. Each agent, then, works toward decreasing its local cost function, while keeping its variable equal to those of neighboring agents. An important line of research focuses on developing algorithms to solve this so called {\it consensus} problem \cite{we13, ch16,mota2013d, bo11}}.\par The iterations of a distributed optimization algorithm can run either synchronously or asynchronously. \ss{The agents in a synchronous iterative algorithm can only update their local iterate at predetermined times and must wait for the slowest agent to finish  before proceeding to the next iteration. Thus, they need to have access to a central clock/coordinator, which is not realistic in the distributed setting. In asynchronous implementations, however, the agents update randomly in time using partial and local information and do not need a central coordinator.} One category of asynchronous algorithms called {\it totally asynchronous} can tolerate arbitrary delays in computation and communication, while the other category, {\it partially asynchronous} algorithms, only work under bounded delay assumptions \cite{ts86,be89}. 
\par\sf{In this paper, we propose a totally asynchronous distributed algorithm to solve a variation of the consensus problem. In our asynchronous setting, agents are {\it active}  based on their local clocks and update using possibly outdated information. In order to achieve fast convergence, we employ the second order information to update the iterate.}
\sm{\subsection{Related Work}}
The field of distributed optimization is pioneered by works in \cite{be89} and \cite{tsitsiklis1984problems}. More recently, various synchronous distributed optimization algorithms have been introduced to solve the consensus problem. One class of these algorithms includes primal first order (sub)gradient descent methods \cite{nedic2009distributed,jakovetic2014fast,matei2011performance,shi2015extra,sun2016distributed,di2016next}, gossip based averaging algorithms \ss{based on pairwise information exchange} \cite{bo6,dimakis2010gossip}, coordinate descent methods \cite{richtarik2016parallel,fercoq2014fast} and dual averaging algorithms \cite{du12,tsianos2012push}. Another line of distributed optimization is based on dual decomposition techniques and Alternating Direction Method of Multipliers (ADMM) \cite{bo11,mota2013d,wei2012distributed}. The last category is the Newton-based methods, where the second order information is used to achieve faster convergence \cite {mo15,wei2013distributed,jadbabaie2009distributed, zanella2011newton, mok15}. \ss{In particular, the network Newton algorithm presented in \cite{mo15,mok15} motivated our work in this paper. Network Newton algorithm is a Newton-based distributed synchronous method, which uses the truncated Taylor's series to approximate the Hessian inverse.}
\par Our work in this paper is mostly related to the literature on asynchronous optimization algorithms. We briefly describe some of the key ideas in this area of research. \ss{One main category is the primal gradient-based algorithms.} The authors in \cite{ra10} presented an asynchronous gossip algorithm to solve the consensus problem. In their asynchronous gossip algorithm, each agent has a local Poisson clock. When the clock ticks, the agent becomes active and averages its estimate with a random neighbor and then adjusts the average using the gradient of its local objective. The authors proved almost sure convergence of their algorithm for convex and nonconvex objective functions under the assumption of uncoordinated diminishing stepsizes, which are related to agents' local clocks. Gossip-based algorithms require bidirectional communication between the agents, which is a bottleneck for some applications like wireless networks. The \sf{authors in \cite{ne11} proposed an alternative, which removes this requirement by using random (unidirectional) broadcast and allowing random link failures in agents' communication.} 
The authors proved almost sure convergence of the asynchronous broadcast-based algorithm to the optimal with diminishing stepsize and to a neighborhood of the optimal point while using constant stepsize. Another work in \cite{assran2018asynchronous}, which focuses on solving the consensus problem over a directed graph, presents a subgradient-push algorithm, in which the agents work asynchronously of the others. The authors showed that a subsequence of the iterates at each agent converges to a neighborhood of the global minimum and that the convergence to the global minimizer can be achieved if all the agents work at the same rate. A distributed asynchronous stochastic optimization algorithm has been introduced in \cite{srivastava2011distributed} to solve a constrained version of the consensus problem. The authors established almost sure convergence for their proposed algorithm. In \cite{pe16}, an algorithmic framework for asynchronous parallel coordinate updates, ARock, has been proposed to find a fix point of a non-expansive operator. At each step of the proposed algorithm, an agent updates a randomly selected coordinate \ss{using a non-expensive mapping}. The authors in \cite{pe16} proved that under the assumption of bounded delays, the algorithm converges to a solution almost surely and for quasi-strongly monotone operators, it converges with a linear rate. 
\par Another strand of the asynchronous distributed optimization literature is based on first order primal-dual schemes. 
The authors in \cite{bi15}, developed a randomized primal-dual optimization algorithm using the idea of stochastic coordinate descent and utilized it to solve the distributed optimization problem asynchronously. \ss{The proposed algorithm, DAPD, converges almost surely under the assumption of independent and identically distributed updates.} The authors in \cite{we13} proposed an asynchronous decentralized algorithm based on the classical Alternating Direction Method of Multipliers (ADMM). In their proposed asynchronous scheme, at each iteration, a random constraint is selected, which in turn selects the corresponding components of decision variable. 
The authors proved that the primal iterates generated by asynchronous ADMM algorithm converges almost surely to an optimal solution with a guaranteed convergence rate of $O(\frac{1}{k})$. Another asynchronous distributed ADMM method has been introduced 
in  \cite{cha16} to solve the consensus problem over a network with a master-worker star topology. In the proposed partially asynchronous setting, the master can update using the information from a subset of the workers and the workers updates do not need to be synchronized. The authors proved that for general nonconvex problems, the algorithm converges to a set of KKT points if the algorithm parameters are chosen based on the network delay. In their follow up work \cite{ch16}, the authors showed that under the assumption of strong convexity, the difference between the augmented Lagrangian and the optimal function value converges to zero with a linear rate. The authors in \cite{hajinezhad2016nestt} proposed a primal-dual method, NESTT, for nonconvex distributed stochastic optimization \ss{over a network with a star graph}. One variation of their algorithm, NESTT-E, can be considered as an asynchronous algorithm in the sense that at each iteration, the master sends information to a randomly selected agent and the agent updates its local primal and dual variables accordingly. The proposed algorithm converges almost surely to a stationary point with a sublinear rate. \re{Recently, the authors in \cite{wu2018decentralized} proposed an asynchronous primal-dual algorithm for decentralized consensus optimization with convex and possibly nondifferentiable objective functions. The authors proved that their algorithm converges to the exact solution under both bounded and unbounded delay assumptions.}
\par This paper is closely related to the literature on asynchronous Newton-based algorithms \cite{ei16,bof2017Newton,bajovic2017distributed}. The authors in \cite{ei16} proposed a distributed partially asynchronous quasi-Newton algorithm to solve a penalized version of the consensus problem, where the convex objective functions have bounded Hessian matrices. This algorithm uses a distributed variation of BFGS to approximate the curvature information. The authors established linear rate of convergence for the proposed algorithm. The recent work in \cite{bajovic2017distributed} incorporates the idling mechanism in distributed second order methods. 
The authors proved that for strongly convex objectives, if the agents' activation probabilities converge to one, then the algorithm converges almost surely and it converges with a R-linear rate, if the activation probabilities converge to one with a geometric rate. Recently, in \cite{bof2017Newton}, a Newton-Raphson consensus algorithm is presented for peer-to-peer optimization which is robust to packet losses. The authors proved that their algorithm is locally \ss{geometrically} convergent. 

\sm{\subsection{Our Contribution}}
Although some of the asynchronous distributed algorithms guarantee sublinear or linear convergence rates, to the best of our knowledge, there is no asynchronous distributed optimization algorithm with superlinear convergence rate. In this paper, we consider solving a penalized version of consensus problem to be able to employ the unconstrained optimization techniques.  We focus on developing an asynchronous algorithm for solving this problem under the assumption of bounded Hessian matrices for convex objective functions. Our contribution is to propose a totally asynchronous (with arbitrary delay) Newton-based algorithm, which converges almost surely and achieves global linear and local quadratic rate of convergence in expectation. \sf{More precisely, we prove that the iterates generated by our algorithm approach the optimal value with a quadratic rate within a certain interval.} To obtain superlinear rate, we build our algorithm on the second order methods and the existing literature on distributed Newton method \cite{wei2013distributed, mo15,jadbabaie2009distributed}. \ss{The main challenge in developing distributed Newton-based methods is to compute the Newton direction, which involves the Hessian inverse and cannot be computed in a distributed way directly. Our asynchronous method employs the matrix splitting technique in the literature \cite{cottle1992linear,saad2003iterative,wei2013distributed, mo15} to replace the Hessian inverse with an approximation \cite{gi5,sh14,bertsekas1983projected}}. \par Our paper builds upon the network Newton algorithm presented in \cite{mo15,mok15}. The authors in \cite{mo15,mok15} proved that the iterations of their algorithm converge linearly and go through a quadratic convergence phase \sf{as long as the stepsize of the updates is smaller than some value related to the optimum of the objective function.} The major difference of our approach lies in the novel asynchronous implementation that requires very different analysis tools. \ss{Moreover, we present a different stepsize selection criteria, which is not related to the optimal function value and depends on the activation probabilities of the agents.} 
For the asynchronous implementation, we consider a setting in which the agents are active and update their corresponding variables with different probabilities. \ss{We assume the agents have access to local buffers, which stores the information from their neighbors. In our algorithm, only one agent is active at each iteration, reads the most recent information from its buffer, and carries out the update. The active agent then broadcasts the updated information to its neighbors.} Unlike the algorithm presented in \cite{bajovic2017distributed}, we do not require the activation probabilities to converge to one. Rather, we assume the agents to be active based on a time invariant and not necessarily uniform probability distribution. We have studied the setting with equal activation probabilities in our previous work in \cite{8264076}, which is a special case of the setting in this paper.\par
The rest of this paper is organized as follows: Section II describes the problem formulation. Section III presents the asynchronous network Newton algorithm. Section IV contains the convergence analysis. Section V presents the simulation results that show the convergence speed improvement of our algorithm compared to the existing methods. Section VI contains the concluding remarks. \\
\noindent\textbf{Basic Notation and Notions:}
A vector is viewed as a column vector. For a matrix $A$, we write $A_{ij}$ to denote the component of $i^{th}$ row and $j^{th}$ column. \sm{We denote by $\mu_{min}(A)$ and $\mu_{max}(A)$ the smallest and largest eigenvalues of a symmetric matrix $A$. Also, for a symmetric matrix $A$, $aI\preceq A\preceq bI$ means that the eigenvalues of $A$ lie in $[a,b]$ interval. For two symmetric matrices $A$ and $B$ we use $A\preceq B$ if and only if $B-A$ is positive semidefinite.} 
For a vector $x$, $x_i$
denotes the $i^{th}$ component of the vector.
We use $x'$ and $A'$ to
denote the transpose of a vector $x$ and a matrix $A$ respectively.
 We use standard Euclidean norm (i.e., 2-norm) unless otherwise noted, i.e., for a vector $x$ in $\mathbb{R}^n$, $\norm{x}=\left(\sum_{i=1}^n x_i^2\right)^{\frac{1}{2}}$. The notation $\mathbbm{1}$ represents the vector of all $1's$ and notation $\textbf{0}$ denotes zero matrix. For a real-valued function $f:\mathbb{R}\rightarrow \mathbb{R}$, the gradient vector and the Hessian
 matrix of $f$ at $x$  are denoted by $\nabla f(x)$ and
 $\nabla^2 f({x})$ respectively.
\section{Problem Formulation}\label{sec:model}
We consider the setup where $n$ agents are connected by an undirected static graph $\mathcal{G(V,E)}$ with $\mathcal{V}$ and $\mathcal{E}$ being the set of vertices and edges respectively. \sm{We denote by $\mathcal{N}_{i}$ the set of neighbors of agent $i$ in the underlying network, i.e., $j\in\mathcal{N}_i$  if and only if $(i,j)\in\mathcal{E}$.} The system-wide goal is to collectively solve the following  problem:
\begin{equation}
\begin{aligned}
 \min_x\quad \frac{1}{2}x^{\prime}(I-W)x+\alpha\sum_{i=1}^{n}f_{i}(x_{i})\,,\label{optFormulation}
\end{aligned}
\end{equation}
\noindent \re{where each function $f_{i}:\mathbb{R}\to\mathbb{R}$ is twice differentiable and convex. Matrix $I$ is the identity matrix of size $n$ by $n$, $x=[x_{1},\,x_{2},\,...,\,x_{n}]^\prime\in\mathbb{R}^{n}$, $\alpha>0$ is a positive scalar,  and the consensus matrix $W\in\mathbb{\mathbb{R}}^{n\times n}$ is a symmetric nonnegative matrix with the
following properties:}
\begin{equation*}
\begin{aligned}
W'& =W,\quad & W\mathbbm{1}&=\mathbbm{1},\quad
\mbox{null}\{I-W\}=\re{\mbox{span}\{\mathbbm{1}\}},\quad & 0\leq& W_{ij}
<1\,.\label{eq:3}
\end{aligned}
\end{equation*}  
Moreover, matrix $W$ represents the network topology, where $W_{ij}\neq 0$ if and only if agents $i$ and $j$ are connected in the underlying network graph. In \sm{our} distributed setting, each agent $i\in\{1, 2, ...,n\}$ has access to a local decision variable $x_{i}\in\mathbb{R}$, its local cost function $f_i$,  and local positive weights $W_{ij}$ for $j$ in $\mathcal{N}_i$, and can communicate with its neighbors defined by the graph. \re{ We denote by $F: \mathbb{R}^n\to\mathbb{R}$ the objective function}, i.e., 
\be\label{eq:defF}F(x) = \frac{1}{2}x^{\prime}(I-W)x+\alpha\sum_{i=1}^{n}f_{i}(x_{i}). \ee

\re{We study problem (\ref{optFormulation}) , because it can be viewed as an approximation to a constrained distributed optimization problem, where the
objective function is a sum of local convex cost functions, i.e.,}
\begin{equation}
\begin{aligned}
\min_{x}&\quad\sum_{i=1}^{n}f_{i}(x_{i})\,,\\ \mbox{s.t.}&\quad x_{i}=x_{j}\,\,\,\forall\,(i,\,j)\,\in\mathcal{E}\,.\label{consensusFormulation}
\end{aligned}
\end{equation}
\re{\ss{Problem (\ref{consensusFormulation}) is the equivalent distributed formulation of the problem $\min_x\sum_{i=1}^nf_i(x)$, which appears in different applications such as machine learning, sensor networks, and wireless systems.} The term $\frac{1}{2}x^{T}(I-W)x$ in problem \req{optFormulation} is equivalent to the penalty on constraint violation in problem \req{consensusFormulation}, because any $x = [x_i]_i$ feasible to problem (\ref{consensusFormulation}) satisfies $Wx = Ix$. The scalar $\alpha$ represents the weight of objective function relative to penalty on constraint violation. In this paper we focus on solving problem \req{optFormulation} considering a fixed penalty constant, $\alpha$. We note that for a fixed $\alpha$ the solutions of problems \req{optFormulation} and \req{consensusFormulation} are not the same and the gap between the solutions is of $O(\alpha)$.  Convergence to the solution of problem \req{consensusFormulation} can be achieved by decreasing the penalty constant \cite{yuan2016convergence,nocedal2006nonlinear}. 
 }
\begin{remark}
For representation simplicity, we focus on the case where $x_i$ is in $\mathbb{R}$. Our results in this paper can be easily generalized to multidimensional case.
\end{remark}
 \sf{We denote by $x^*$ the minimizer of problem (\ref{optFormulation}) and by $F^*=F(x^*)$ the minimum objective function value.} 
 We adopt the following standard assumptions on problem \eqref{optFormulation}.
\begin{assumption}[Bounded Hessian]\label{assm:BoundedHessian} The local objective functions $f_{i}(x)$ are convex,
	twice continuously differentiable with bounded Hessian, i.e. for all $x_i$ in $\mathbb{R}$\[
	0<m\leq\nabla^{2}f_{i}(x_i)\leq M<\infty.\]
\end{assumption}
\begin{assumption} [Lipschitz Hessian]\label{assm:LipHessian} The Hessian matrices of local objective functions, $\nabla^{2}f_{i}(x_i)$, are L-Lipschitz continuous, i.e., for all $x_i,\,\bar{x}_i$ in $\mathbb{R}$,
\[	\norm{ \nabla^{2}f_{i}(x_i)-\nabla^{2}f_{i}(\bar{x}_i)} \leq L\norm{ x_i-\bar{x}_i}.\]
\end{assumption}
\begin{assumption} [Bounded Consensus Matrix Weight]\label{assm:Consensus} There exist positive scalars $\delta$ and $\Delta$ with $0<\delta\leq \Delta<1$, such that the diagonal elements of the consensus matrix $W$ satisfy
\[	\delta\leq W_{ii}\leq\Delta,\,\,\,\,i=1,\,2,\,...,\,n\,.\]\end{assumption}
The first assumption requires that the eigenvalues of the Hessian matrix are bounded with two positive numbers, which is true if and only if the objective functions are $m-$strongly convex \sm{and have $M-$Lipschitz gradients.} The second assumption states that the Hessian does not change too fast. Both of these assumptions are standard conditions on the local objective functions for developing Newton-based algorithms \cite{boydbook}. The last assumption on matrix $W$ is satisfied by many standard choices of consensus matrices \cite{nedic2009distributed,ts86,xiao2005scheme}, \sm{we note that, considering the definition of matrix $W$, the upper bound on the diagonal elements $\Delta$ is guaranteed to be less than one.}
 
These assumptions hold in this paper and our goal is to design an asynchronous distributed Newton-based algorithm, with superlinear rate of convergence, to solve problem \req{optFormulation}.
\section{Asynchronous Network Newton Method}\label{sec:alg}

Our asynchronous algorithm is based on Newton's method for unconstrained problem with the following iteration
\[
x(t+1)=x(t)+\varepsilon d(t),
\]
where the notation $(t)$ indicates the iteration \sm{count}, $\varepsilon$ is some positive stepsize and $d(t)$ is the Newton direction which is equal to \[d(t)=-H(t)^{-1}g(t),\]  with $g$ and $H$ being the gradient and Hessian of objective function respectively, i.e., $g(t)=\nabla F(x(t))$ and $H(t)=\nabla^{2}F(x(t))$. By using the definition of function F [c.f.\ Eq. \req{eq:defF}], we have that each component of gradient $g$ is given by
\be\label{eq:gradient} g_i(t)=[(I-W)x(t)]_i+\alpha \nabla f_{i}(x_{i}(t)).\ee The Hessian matrix $H$ can be written as
\be \label{eq:defH} H(t)=I-W+\alpha G(t),\ee
 where $G(t)\in\mathbb{R}^{n\times n}$ is a diagonal matrix
with 
\be\label{eq:defG}G_{ii}(t)=\nabla^{2}f_{i}(x_{i}(t)).\ee 

\subsection{Background on Approximation of the Newton Direction}
\sm{In this section}, we first outline the method used in \cite{mo15} to solve the same problem in a synchronous distributed way, we then introduce our asynchronous version of this algorithm. The authors of \cite{mo15} represented the Hessian inverse as a convergent series of matrices, where each of the terms can be computed locally. The algorithm approximates the inverse of Hessian matrix by using a finite truncated summation of the terms. 

The Hessian matrix $H$ [c.f.\ Eq. \eqref{eq:defH}] is splitted as follows,
\be\label{eq:Hsplit}
H(t)=D(t)-B,\ee
with
\be
D(t)=\alpha G(t)+2(I-W_{d}),
\quad B=I-2W_{d}+W,\label{eq:defDB}\ee
where  $W_{d}$ is a diagonal matrix with $[W_d]_{ii} = W_{ii}$. 
Matrix $G(t)$ is a positive definite matrix because of  the assumption that the local functions have bounded second derivative [c.f.\ Assumption \ref{assm:BoundedHessian}]. By Assumption \ref{assm:Consensus}, $[W_d]_{ii}=[W]_{ii}<1$ and thus $I-W_{d}$ is also positive definite. Therefore, the diagonal matrix $D(t)$ is positive definite and thus invertible. By factoring $D(t)^{1/2}$ on both sides of Eq.\ \eqref{eq:Hsplit}, we have 
\[ H(t) = 
D(t)^{1/2}\big(I-D(t)^{-1/2}BD(t)^{-1/2}\big)
D(t)^{1/2},\] which implies that 
\[ H(t)^{-1} = 
D(t)^{-1/2}(I-D(t)^{-1/2}BD(t)^{-1/2})^{-1}
D(t)^{-1/2}.\]The middle inverse term can be written as
\begin{align*}\big(I-D(t)^{-1/2}BD(t)^{-1/2}\big)^{-1} = \sum_{k=0}^\infty \big(D(t)^{-1/2}BD(t)^{-1/2}\big)^k,\end{align*} whenever spectral radius (largest eigenvalue by magnitude) of matrix $D(t)^{-1/2}BD(t)^{-1/2}$ is strictly less than 1, \re{Chapter 5.6 of} \cite{horn1985matrix}. Using the particular structure of matrices $D$ and $B$, the following lemma from \cite{mo15} guarantees that the spectral radius of matrix $D(t)^{-1/2}BD(t)^{-1/2}$ is strictly less than 1.
\begin{lemma}  \label{proposition 2} Under Assumptions \ew{\ref{assm:BoundedHessian} and \ref{assm:Consensus}}, $D(t)^{-1/2}BD^{-1/2}$
	is positive semi-definite
	with bounded eigenvalues which are strictly less than 1, i.e., 
	\begin{equation*}
	\begin{aligned}
	\boldsymbol{0}\preceq D(t)^{-1/2}BD(t)^{-1/2}\preceq\rho I\,,\label{eq:26}
	\end{aligned}
	\end{equation*}
 where $\rho=2(1-\delta)/(2(1-\delta)+\alpha m)<1$.
\end{lemma}
 Hence, the Hessian inverse is equal to 
\begin{align}\label{eq:inverse}&
H(t)^{-1}=D(t)^{-1/2}\sum_{k=0}^{\infty}\big(D(t)^{-1/2}BD(t)^{-1/2}\big)^{k}D(t)^{-1/2}.
\end{align}
Therefore, the Newton direction can be written as 
\begin{align}\label{eq:dCom}& d(t) = -D(t)^{-1/2}\sum_{k=0}^{\infty}\big(D(t)^{-1/2}BD(t)^{-1/2}\big)^{k}D(t)^{-1/2} g(t).\end{align}

We now check the distributed implementation of the above equation following the same analysis as in \cite{mo15}. We note that each of the diagonal elements of $D(t)$ can be computed locally at each node $i$ as 
\begin{equation*}
D_{ii}(t)=\alpha\nabla^{2}f_{i}(x_{i}(t))+2(1-W_{ii}).
\end{equation*} Moreover, elements of matrix $B$ satisfy
\[B_{ii}=1-2W_{ii}+W_{ii} = 1-W_{ii},\quad B_{ij}=W_{ij},\]  which can also be computed using local information available to agent $i$. The multiplication by diagonal matrix $D(t)^{-1/2}$ is effectively scaling using local information and multiplication of matrix $B$ corresponds to communicating with immediate neighbors, and both can be carried out locally. The $k^{th}$ order term in Eq. \eqref{eq:dCom}, can be computed via $k$ local neighborhood information exchanges, i.e., information from neighbors of $k-$hop away. Hence, the Newton direction $d$ can be computed using local information. However, due to the computation limitation, \sm{the authors in \cite{mo15} proposed to} truncate the series to include only finite number of terms and form an approximation of the Newton direction, which results in the network Newton algorithm presented in \cite{mo15}. 

\subsection{Asynchronous Network Newton}

Based on the \sm{results} from the previous section, we can now develop our asynchronous network Newton algorithm. We assume that at each iteration $t$, each agent $i$ is active with probability $p_i$. \sm{The active agent} updates its corresponding variable using local information and information from immediate neighbors to compute its local Newton direction. We assume that each agent is active infinitely often in time. When we are only concerned with the total number of updates (instead of total time elapsed), we can equivalently count the number of iterates by increasing the iteration counter by one, whenever any agent is active. \sf{ We emphasize that each agent does not need a counter of the iteration number. Instead, it simply needs to maintain the most updated information of itself and its neighbors.} \re{We assume that one$-$hop neighbors of the active agent are notified and can perform some basic computations.} \sf{When an agent is not active, we assume that it may still receive information. \begin{algorithm}
\caption{Asynchronous Network Newton}\label{async NN}
\begin{algorithmic}[1]
\State Initialization:  
For $i=1, 2, ..., n$, each agent $i$ sets $x_{i}(0)=0$, computes $D_{ii}(0), g_i(0), d_i(0), B_{ii}, B_{ij}$:
\[D_{ii}(0)=\alpha\nabla^{2}f_{i}(x_{i}(0))+2(1-W_{ii}),\quad
g_{i}(0)=(1-W_{ii})x_{i}(0)+\alpha\nabla f_{i}(x_{i}(0)),\] 
\[d_{i}^{(0)}(0)=-D_{ii}(0)^{-1}g_{i}(0),\quad
B_{ii}=1-W_{ii}, \quad B_{ij}=W_{ij},\]
and broadcasts $d_i^{(0)}(0)$ and stores received $d_j^{(0)}$, $x_j$ values from neighbors and determines stepsize parameter $\varepsilon$.
\For{$t=1,2,...$}
\State An agent $i\in \left\{ 1,2,...,n\right\}$ is active according to its local clock with probability $p_i$.
\State Active agent $i$ computes  $g_i(t-1)$, $d_{i}^{(0)}(t-1)$ \ew{ and the local Newton direction $d_i(t-1)$ using the most recent information from neighbors, $x_j(t-1)$ and $d_j^{(0)}(t-1)$ for $j$ in $\mathcal{N}_i$} as 
\[g_{i}(t-1)=(1-W_{ii})x_{i}(t-1)+\alpha\nabla f_{i}(x_{i}(t-1))-\sum_{j\in\mathcal{N}_{i}}W_{ij}x_{j}(t-1),\]
\[d_{i}^{(0)}(t-1)=-D_{ii}(t-1)^{-1}g_{i}(t-1),\]
\[d_{i}(t-1)=D_{ii}(t-1)^{-1}\bigl[B_{ii}d_{i}^{(0)}(t-1)-g_{i}(t-1)+\sum_{j\in\mathcal{N}_{i}}B_{ij}d_{j}^{(0)}(t-1)\bigr]
.\]
\State Active agent $i$ takes a Newton step and updates its local iterate by
\[x_{i}(t)=x_{i}(t-1)+\frac{\varepsilon}{p_i} d_{i}(t-1).\]
\State Active agent updates $D_{ii}(t),$ $ g_i(t),$ and $d_i^{(0)}(t)$ by  
\[D_{ii}(t)=\alpha\nabla^{2}f_{i}(x_{i}(t))+2(1-W_{ii}),\]
\[g_{i}(t)=(1-W_{ii})x_{i}(t)+\alpha\nabla f_{i}(x_{i}(t))-\sum_{j\in\mathcal{N}_{i}}W_{ij}x_{j}(t-1),\]
\[d_{i}^{(0)}(t)=-D_{ii}(t)^{-1}g_{i}(t).\]
\State Active agent $i$ broadcasts $d_{i}^{(0)}(t)$ and $x_i(t)$ to its neighbors. 
\State \re{All agents $j\in\mathcal{N}_i$, listen and store received $d_i^{(0)}(t)$ and $x_i(t)$, update $g_j(t)$ similar to step $4$ and $d_j^{(0)}(t)$ similar to step $6$, and broadcast $d_j^{(0)}(t)$ to their neighbors.} 
\State \re{All inactive agents $l\in\mathcal{N}_j$ passively listen and store received  $d_j^{(0)}(t)$ values from $j\in\mathcal{N}_i$. All other variables remain at their previous values.}
\EndFor
\State \textbf{end for}
\end{algorithmic}
\end{algorithm} This can be achieved by maintaining a buffer for each agent in which the old information is overwritten whenever new information is received from the neighbors. When an agent is active, it reads the most recent information from its buffer.} \par In order to take into account the different activation probabilities, we assume that each agent's stepsize is inversely proportional to its activation probability, which essentially means that the agent that is active less often, uses bigger stepsize. One way to implement this process is to assume that each agent is associated with a Poisson clock, which ticks according to a Poisson process. \ss{Having Poisson clocks is a standard assumption in implementing asynchronous algorithms \cite{ra10, ne11,pe16}.} The clocks do not need to have same parameters and they are independent from each other. In this case, we can assume that in the initialization step the agents communicate their Poisson rates, so that the summation of the rates is known to all agents. Therefore, each agent can compute its activation probability by dividing its own rate by the summation of the rates and determines its stepsize accordingly.
\par\ss{We assume that only one clock ticks at each iteration, which is a natural assumption for the Poisson clocks, and also the clock activation happens on a slower time scale than the agents update. These assumptions  imply that only one agent is active at each iteration and finishes the update before another activation happens.} This type of asynchronous algorithm is also known as {\it randomized} algorithm. Our algorithm is totally asynchronous, in the sense that it does not assume each agent updates at least once within a certain bounded number of iterations \cite{be89}.
We adopt the following assumption on activation probabilities. 
\begin{assumption}\label{ass:activeprob}
The activation probabilities for all agents $i\in\{1, 2, ..., n\}$ satisfy
\[0<\pi\leq p_i\leq\Pi<1.\] 
\end{assumption}
\sm{We note this assumption is automatically satisfied due to the fact that every agent updates infinity often in time.} We also have $\sum_{i=1}^np_i=1$. 
We denote by $P$ the \sf{time invariant} diagonal matrix with the diagonal elements equal to the probabilities $p_i$, so we have
\[\pi I\preceq P\preceq \Pi I.\]
\sm{By the nature of the asynchronous distributed algorithm, we can only compute the $0^{th}$ and $1^{st}$ order terms in the Hessian inverse formula, [c.f Eq. (\ref{eq:inverse})].} We denote by $\hat H(t)^{-1}$ the approximation of Hessian inverse using the first two terms of the infinite series, i.e., 
\be\label{eq:defHatH}\hat{H} (t)^{-1}=D(t)^{-1/2}\big[I+D(t)^{-1/2}BD(t)^{-1/2}\big]D(t)^{-1/2}.\ee
Resulting in the Newton direction approximation defined by 
 \be\label{eq:defd}d(t) = -\hat{H}(t)^{-1}g(t).\ee The asynchronous network Newton algorithm is given in Algorithm \ref{async NN}. We note that using Eq. (\ref{eq:defHatH}) and Eq. (\ref{eq:defd}), the Newton step in our algorithm can be expressed as \[d(t)=-D(t)^{-1}g(t)-D(t)^{-1}BD(t)^{-1}g(t).\] We denote by $d^{(0)}(t)$  the Newton direction in which the Hessian matrix is approximated using the $0^{th}$ order term of the Taylor's expansion, i.e., $d^{(0)}(t)=-D(t)^{-1}g(t)$ . Therefore, the Newton direction is equal to\[d(t)=D(t)^{-1}\big(Bd^{(0)}(t)-g(t)\big).\]
Note that $D(t)$ is diagonal and $B$ is representing the underlying graph of the network, the Newton direction for each agent can be written as \[d_{i}(t)=D_{ii}(t)^{-1}\bigl[B_{ii}d_{i}^{(0)}(t)-g_{i}(t)+\sum_{j\in\mathcal{N}_{i}}B_{ij}d_{j}^{(0)}(t)\bigr],
\]where $g_i(t)$ is computed using Eq. (\ref{eq:gradient}).\par We next verify that the algorithm can indeed be implemented in an asynchronous distributed way. In this algorithm, in the initialization step, each agent $i$ computes $D_{ii}(0)$, $g_i(0)$ and $d_i^{(0)}(0)$ using local information, broadcasts $x_i(0)$ and $d_i^{(0)}(0)$ and receives those of neighbors by utilizing its own buffer. At each iteration $t$, a random agent $i$ is active with probability $p_i$ \ss{and has access to $W_{ii}$, $W_{ij}$, $x_i(t-1)$, $\nabla f_i(x_i(t-1))$, $d_i^{(0)}(t-1)$, $D_{ii}(t-1)$, and also $x_j(t-1)$, and $d_j^{(0)}(t-1)$ from its neighbors $j\in\mathcal{N}_i$. Then in step $4$ of Algorithm \ref{async NN}, the active agent $i$ computes $g_i(t-1)$ using the local information $W_{ii}$, $W_{ij}$, $x_i(t-1)$, and $\nabla f_i\big(x_i(t-1)\big)$, and $x_j(t-1)$ from its neighbors. Then it computes $d_i^{(0)}(t-1)$ using $D_{ii}(t-1)$ and $g_i(t-1)$, and uses the updated $d_i^{(0)}(t-1)$ and also the most recent $d_j^{(0)}(t-1)$ form $j\in\mathcal{N}_i$ to compute the approximated Newton direction $d_i(t-1)$. The active agent computes the next iterate $x_i(t)$ in step $5$ and uses the new $x_i(t)$ to update $D_{ii}(t)$, $g_i(t)$, and $d_i^{(0)}(t)$ in step $6$. Once the active agent finishes its iterate, it broadcasts updated information $d_i^{(0)}(t)$ and $x_i(t)$ to its neighbors in step $7$. The agent $j\in\mathcal{N}_i$ receives this information from active agent $i$, \re{updates $g_j(t)$ and $d_j^{(0)}(t)$ using the new information} and keeps previous values of $D_{jj}(t-1)$ and $x_j(t-1)$. \re{Agent $j\in\mathcal{N}_i$ broadcasts its most recent $d_j^{(0)}(t)$ to its neighbors. We note that in this implementation, one$-$hop neighbors of the active agent, i.e., $j\in\mathcal{N}_i$, are not completely passive. They are notified by the active agent and update $g_j(t)$ and $d_j^{(0)}(t)$.}} 
\sf{\begin{remark} \label{unifrom_case}
One special case is the uniform activation, in which all the agents are active with equal probabilities. In this case, there is no need to scale the agents' stepsizes with the inverse of their activation probabilities, i.e., in step $5$ of Algorithm \ref{async NN}, the active agent use $\varepsilon$ instead of $\frac{\varepsilon}{p_i}$. This case is studied in \cite{8264076} and all the results there, are special cases of our convergence analysis in this paper.
\end{remark}}
\section{Convergence Analysis}\label{sec:conv}
In this section, we present some existing preliminaries in Section \ref{sec:prel}, which we use to show almost sure and global linear rate of convergence of the proposed asynchronous method in Section \ref{sec:convAsync} and also to establish local quadratic rate of convergence (in expectation) in Section \ref{sec:Quad}. 

\subsection{Preliminaries}\label{sec:prel}
We state three lemmas which are adopted from synchronous network Newton method proposed in \cite{mo15}. These lemmas \ew{have been proven} in \cite{mo15} \ew{only using} the properties of the local objective functions and the consensus matrix $W$ \ew{and are not dependent to the algorithm implementation. We restate them here for completeness.}
\begin{lemma} \label{lemma 01}  If Assumption \ref{assm:LipHessian} holds, then for every $x, \bar{x}\in\mathbb{R}^{n}$\ew{,} the Hessian matrix, $H(x)=\nabla^{2}F(x)$ , is $\alpha L$-Lipschitz
continuous, i.e.,

\begin{equation*}
\norm{ H(x)-H(\bar{x})} \leq\alpha L\norm{ x-\bar{x}}.
\end{equation*}
\end{lemma}

\begin{lemma} \label{lemma 1} If Assumptions \ref{assm:BoundedHessian},\ref{assm:LipHessian} and \ref{assm:Consensus} hold, starting from any initialization, the eigenvalues
of $H(t)$ , $D(t)$ , and $B$ [c.f.\ Eqs.\ \eqref{eq:defH}, \eqref{eq:defDB}] are bounded for all $t$ by
\begin{align*}
\alpha mI\preceq & H(t) \preceq(2(1-\delta)+\alpha M)I,\\
(2(1-\Delta)+\alpha m)I&\preceq D(t) \preceq(2(1-\delta)+\alpha M)I,\\
\boldsymbol{0} &\preceq B\preceq 2(1-\delta)I.
\end{align*}
\end{lemma}
\begin{lemma}  \label{lemma:lemma 4.2} \re{Recall the definition of $\rho$ from Lemma \ref{proposition 2},} under Assumptions \ref{assm:BoundedHessian} and \ref{assm:Consensus}, starting from any initialization, the eigenvalues of the Hessian inverse approximation [cf.\ Eq.\ \eqref{eq:defHatH}] are bounded for all $t$ by
\[
\lambda I\preceq\hat{H}(t)^{-1}\preceq\Lambda I\,,
\]
where 
$
\varLambda=\frac{1+\rho}{2(1-\Delta)+\alpha m}, \lambda=\frac{1}{2(1-\delta)+\alpha M}.$
\end{lemma}
The next lemma from \cite{po87,ra10} is used to establish almost sure convergence of the asynchronous network Newton algorithm.
\begin{lemma}  \label{lemma:lemma 4.5} Let $\big(\Omega,\mathcal{\,F},\,\mathcal{P}\big)$
be a probability space and $\mathcal{F}_{0}\subseteq\mathcal{F}_{1}\subseteq...$
be a sequence sub $\sigma$- fields of $\mathcal{F}$. Let $\left\{ X_{t}\right\} ,\left\{ Y_{t}\right\} ,\left\{ Z_{t}\right\} ,$
and $\left\{ W_{t}\right\}$ be $\mathcal{F}_{t}$ -measurable random
variables such that $\left\{X_{t}\right\}$ is bounded below and $\left\{ Y_{t}\right\} $
, $\left\{ Z_{t}\right\} $ , and $\left\{W_{t}\right\}$ are non-negative. Let $\sum_{t=0}^{\infty}Y_{t}<\infty$
and $\sum_{t=0}^{\infty}W_{t}<\infty$ , and
\begin{equation*}
\begin{aligned}
\mathbb{E}\big[X_{t+1}\textbf{\Big |}\mathcal{F}_{t}\big]\leq(1+Y_{t})X_{t}-Z_{t}+W_{t}\,,\label{eq:29}
\end{aligned}
\end{equation*}
hold with probability 1. Then \fm{with probability 1,} $\left\{ X_{t}\right\}$ converges and $\sum_{t=0}^{\infty}Z_{t}<\infty$.
\end{lemma}


The last two lemmas are adopted from \cite{boydbook}, and are used as key relations in the convergence rate analysis. 

\begin{lemma} \label{lemma:lemma 4.8} If $F:\mathbb{R}^n\to\mathbb{R}$ is a twice continuously differentiable function with
$\kappa$-Lipschitz continuous Hessian, then for any $u, v$ in $\mathbb{R}^n$, we have
\begin{equation*}
\begin{aligned}
\norm{ \nabla F(v)-\nabla F(u)-\nabla^{2}F(u)(v-u)} \leq\frac{\kappa}{2}\norm{v-u}^2.
\end{aligned}
\end{equation*}
\end{lemma}
\begin{lemma} \label{functionbound}
If $F:\mathbb{R}^n\to\mathbb{R}$ is a strongly convex function with $mI\preceq\nabla^2F(y)\preceq MI$ for all $y\in\mathbb{R}^n$, then for any $u, v$ in $\mathbb{R}^n$, we have
\[F(u)\geq F(v)-\frac{1}{2m}\norm{\nabla F(v)}^2\] and\[F^*\leq F(v)-\frac{1}{2M}\norm{\nabla F(v)}^2.\]
\end{lemma}
\subsection{Convergence of Asynchronous Network Newton Algorithm}\label{sec:convAsync}
\sf{In this section, in Theorem \ref{thm:conve}, we show that the sequence of function values $\left\{ F(x(t))\right\}$ generated by the asynchronous network Newton algorithm converges to $F^*$ almost surely. We also show that the function values $F(x(t))$ and the iterates $x(t)$ converge to $F^*$ and $x^*$ with a global linear rate in expectation in Theorem \ref{linconv}.  In what follows, we introduce some notation used to connect asynchronous and synchronous algorithms. To model the asynchrony, we define a stochastic diagonal {\it activation matrix} $\Phi(t)$ in $\mathbb{R}^{n\times n}$ by}
\begin{equation}
\begin{aligned}
\Phi(t)_{ii}=\begin{cases}
1 & \mbox{if i is active at time t,}\\
0 & \mbox{otherwise.}
\end{cases}\label{eq:15}
\end{aligned}
\end{equation}
This matrix indicates which agent is active at time $t$. We denote  by $\Phi^i$ a diagonal matrix with its $i^{th}$ element equal to $1$ and the rest equal to zero. This matrix is a realization of the activation matrix $\Phi(t)$. We also use $\mathcal{F}_{t}$ to denote the $\sigma$-field capturing all realizations (activations) of the algorithm up to and including time $t$. \sm{We can now define the asynchronous Newton direction generated by Algorithm \ref{async NN} at iteration $t$, i.e., $d^a(t-1)$ as follows}
\begin{equation}
\begin{aligned}
d_{i}^a(t-1)=\begin{cases}
-\big[\hat{H}(t-1)^{-1}g(t-1)\big]_{i} & \mbox{if i is active at time t,}\\
0 &\mbox{otherwise.}
\end{cases}
\end{aligned}\label{eq:dt-1}
\end{equation}
The asynchronous network Newton update formula can be aggregated as 
\[x(t)=x(t-1)+\varepsilon P^{-1} d^a(t-1)\,.\]
Conditioned on $\mathcal{F}_{t-1}$, we have that $d^a(t-1)$ is a random vector given by 
\[d^a(t-1)=-\Phi(t)\hat{H}(t-1)^{-1}g(t-1),\] where the random matrix $\Phi(t)$ chooses one element of $\hat{H}(t-1)^{-1}g(t-1)$ to keep in $d^a(t-1)$ and makes the rest 0 as in Eq.\ \eqref{eq:dt-1}. Thus, the asynchronous Newton update can be written as
\be\label{eq:xUpdateAsync}x(t)=x(t-1)-\varepsilon P^{-1}\Phi(t)\hat{H}(t-1)^{-1}g(t-1).\ee
We note that at iteration $t$, each agent $i$ is active with  probability $p_i$, thus we have that \be\label{eq:expPhi}\mathbb{E}[\Phi(t)|\mathcal{F}_{t-1}] = \sum_{i=1}^n p_i \Phi^i=P,\ee where the expectation is with respect to the realization of the algorithm. 
\begin{theorem} \label{thm:conve} Consider the iterates $\{x(t)\}$ generated by the asynchronous network Newton algorithm as in Algorithm \ref{async NN}, and recall the definition of $\lambda$ and $\Lambda$ from Lemma \ref{lemma:lemma 4.2} and the notations $g(t)=\nabla F(x(t))$ and $H(t)=\nabla^2 F(x(t))$, then if the stepsize parameter $\varepsilon$
is chosen as 
\begin{equation}
\begin{aligned}
0<\varepsilon\leq 2\pi\big(\frac{\lambda}{\Lambda}\big)^{2}\,,\label{eq:30}
\end{aligned}
\end{equation}
then
\begin{equation}
\begin{aligned}
\mathbb{E}\big[F(x(t))\textbf{\Big |}\mathcal{F}_{t-1}\big]\leq F(x(t-1))-\big(\varepsilon\lambda-\frac{\varepsilon^{2}\Lambda^{2}}{2\lambda\pi}\big)\norm {g(t-1)}^{2}.\label{eq:martingale}
\end{aligned}
\end{equation}
and the sequence $\left\{ F(x(t))\right\} $ converges to the optimal value of problem \ref{optFormulation}, $F^*$, almost surely.
\end{theorem}
\begin{proof} Using the Taylor's theorem, we have that for any $a$, $b$ in $\mathbb{R}^n$, there exists a point $c$ on the line segment between them such that
\[
F(a)= F(b)+g(b)'(a-b)+\frac{1}{2}(a-b)' H(c)(a-b).
\] By using the bound on Hessian matrix in Lemma \ref{lemma 1} we have
\[
F(a)\leq F(b)+g(b)'(a-b)+\frac{2(1-\delta)+\alpha M}{2}(a-b)'(a-b).
\]
Thus, for any realization of the activation matrix, $\Phi(t)$, we can substitute $a=x(t)$, $b=x(t-1)$ and $\lambda=\frac{1}{2(1-\delta)+\alpha M}$ from Lemma \ref{lemma:lemma 4.2}, and have
\begin{equation}
\begin{aligned}
F(x(t))\leq F(x(t-1))+g(t-1)^\prime(x(t)-x(t-1))+\frac{1}{2\lambda}\norm{x(t)-x(t-1)}^{2}\,,\label{eq:31}
\end{aligned}
\end{equation}
From Eq. (\ref{eq:xUpdateAsync}), we have
\begin{equation}
x(t)-x(t-1)=-\varepsilon P^{-1}\Phi(t)\hat{H}(t-1)^{-1}g(t-1)\,.\label{phiterm}
\end{equation}
Taking expectation \sf{on both sides} of (\ref{eq:31}) conditioned on  $\mathcal{F}_{t-1}$ and using (\ref{phiterm}) we get
\begin{equation*}
\begin{aligned}
&\mathbb{E}\big[F(x(t))\textbf{\Big |}\mathcal{F}_{t-1}\big]\leq F(x(t-1))-\varepsilon g(t-1)^\prime P^{-1}\mathbb{E}\big[\Phi(t)\textbf{\Big |}\mathcal{F}_{t-1}\big]\hat{H}(t-1)^{-1}g(t-1)+\\&\frac{\varepsilon^{2}}{2\lambda}\mathbb{E}\Big[\norm{  P^{-1}\Phi(t)\hat{H}(t-1)^{-1}g(t-1)} ^{2}\textbf{\Big |}\mathcal{F}_{t-1}\Big], \label{eq:32}
\end{aligned}
\end{equation*}
where we used the property that conditioned on $\mathcal{F}_{t-1}$, $x(t-1)$, $\hat{H}(t-1)$, and $g(t-1)$ are deterministic.  We note that each agent $i$ is active with probability $p_i$ at iteration $t$; therefore, 
\begin{align*} &\mathbb{E}\Big[\norm{ P^{-1} \Phi(t)\hat{H}(t-1)^{-1}g(t-1)} ^{2}\textbf{\Big |}\mathcal{F}_{t-1}\Big]=\sum_{i=1}^{n}p_i\Big(\frac{1}{p_i} \big[\hat{H}(t-1)^{-1}g(t-1)\big]_{i}\Big)^{2}=\\&\sum_{i=1}^{n}\frac{1}{p_i}\big[\hat{H}(t-1)^{-1}g(t-1)\big]_{i}^{2}\leq\frac{1}{\pi}\sum_{i=1}^{n}\big[\hat{H}(t-1)^{-1}g(t-1)\big]_{i}^{2}=\frac{1}{\pi}\norm {\hat{H}(t-1)^{-1}g(t-1)}^2.
\end{align*}
By using the previous two relations and Eq. (\ref{eq:expPhi}), we have 
\begin{equation*}
\begin{aligned}
\mathbb{E}&\big[F(x(t))\textbf{\Big |}\mathcal{F}_{t-1}\big]\leq F(x(t-1))-\varepsilon g(t-1)^\prime P^{-1} P\hat{H}(t-1)^{-1}g(t-1)\\&+\frac{\varepsilon^{2}}{2\lambda\pi} \norm{ \hat{H}(t-1)^{-1}g(t-1)}^2\,.\label{eq:32}
\end{aligned}
\end{equation*}By using the bounds on the approximated Hessian [c.f. Lemma \ref{lemma:lemma 4.2}], we have
\[-\varepsilon g(t-1)^\prime \hat{H}(t-1)^{-1}g(t-1)\leq -\varepsilon \lambda\norm{g(t-1)}^2,\]and
\begin{align*}\norm {\hat{H}(t-1)^{-1}g(t-1)}^{2}\leq\Lambda^2 \norm{g(t-1)}^2.\end{align*}
Combining the three relations above yields
\begin{equation}
\begin{aligned}
\mathbb{E}\big[F(x(t))\textbf{\Big |}\mathcal{F}_{t-1}\big]\leq F(x(t-1))-\big(\varepsilon\lambda-\frac{\varepsilon^{2}\Lambda^{2}}{2\lambda\pi}\big)\norm {g(t-1)}^{2}. \label{martingale1}
\end{aligned}
\end{equation}

We next argue that the scalar $\varepsilon\lambda-\frac{\varepsilon^{2}\Lambda^{2}}{2\lambda\pi}\geq 0$. We start by rewriting it as 
\[\varepsilon\lambda-\frac{\varepsilon^{2}\Lambda^{2}}{2\lambda\pi} = \frac{2\varepsilon\lambda^2\pi-\varepsilon^2\Lambda^2}{2\lambda\pi} = \frac{\varepsilon(2\lambda^2\pi-\varepsilon\Lambda^2)}{2\lambda\pi}.\] Since the stepsize parameter $\varepsilon$ satisfies the bounds in (\ref{eq:30}), i.e.,
\[\varepsilon\leq 2\pi\big(\frac{\lambda}{\Lambda}\big)^{2},
\]
the scalar $\varepsilon\lambda-\frac{\varepsilon^{2}\Lambda^{2}}{2\lambda\pi}$ is nonnegative. In addition, we have that $F(x(t))$ is strongly convex, thus bounded below by its second order approximation \cite{boydbook}. Therefore, we can use Eq. (\ref{martingale1}) together with the result of Lemma \ref{lemma:lemma 4.5}, with $Y_t = 0$, $W_t=0$, to conclude that the sequence $\left\{ F(x(t))\right\} $
converges almost surely and $\sum_{t=0}^{\infty}\big(\varepsilon\lambda-\frac{\varepsilon^{2}\Lambda^{2}}{2\lambda\pi}\big)\norm {g(t-1)} ^{2}<\infty$, 
\fm{with probability 1}, which means that $\norm {g(t)} ^{2}$ converges to
zero \ew{almost surely}. By combining these two results, we complete the proof.
\end{proof}
\re{\begin{remark} In our algorithm, the stepsize $\varepsilon$ is common among all agents. Computing $\varepsilon$ requires global variables across the network, i.e., $m$, $M$, $\delta$, and $\Delta$, which can be obtained either by applying a consensus algorithm prior to the main algorithm or by estimating global bounds on the properties of the objective function \cite{jadbabaie2009distributed,mo15,shi2015extra,wu2018decentralized}.
\end{remark}}
\begin{theorem} \label{linconv}
\sf{Consider the iterate $\left\{ x(t)\right\} $ generated by the asynchronous network Newton algorithm as in Algorithm \ref{async NN}. If the stepsize parameter
$\varepsilon$ satisfies }
\begin{equation}
\begin{aligned}
0<\varepsilon<\min\,\left\{ \frac{1}{2},\,2\pi\big(\frac{\lambda}{\Lambda}\big)^{2}\right\} \,,\label{eq:step}
\end{aligned}
\end{equation}
\sm{then the sequences $\left\{ F(x(t))\right\}$ and $\left\{ x(t)\right\}$ converge linearly in expectation to their optimal values, i.e.,
\begin{equation*}
\begin{aligned}
\mathbb{E}\big[F(x(t))-F^*\big]\leq\big(1-\beta\big)^{t}\big[F(x(0))-F^*\big]\,,
\end{aligned}
\end{equation*}
and
\begin{equation*}
\begin{aligned}
\mathbb{E}\Big[\norm{x(t)-x^*}\Big]\leq\Bigg(\frac{2\big(F(x(0))-F^*\big)}{\alpha m}\Bigg)^{1/2}\Big((1-\beta)^{1/2}\Big)^t.
\end{aligned}
\end{equation*}}
\noindent where $\beta=\frac{\alpha m\varepsilon(2\pi\lambda^{2}-\varepsilon\Lambda^{2})}{\lambda\pi}$.\\
\end{theorem}
\begin{proof}
We use the result of Theorem \ref{thm:conve} to prove the global linear rate of convergence. We note that our objective function, $F(x)$, is $\alpha m$-strongly convex, then by using the result of Lemma \ref{functionbound}, we have
\begin{equation}
-\norm {g(t-1)} ^{2}\leq-2\alpha m\big(F(x(t-1))-F^*\big).\label{eq:bound}
\end{equation}
By subtracting $F^*$ \sf{on both sides} of (\ref{eq:martingale}) and substituting the bound in Eq. (\ref{eq:bound}), we have
\begin{equation}
\begin{aligned}
\mathbb{E}\big[F(x(t))-F^*\textbf{\Big |}\mathcal{F}_{t-1}\big]\leq\big(1-\beta\big)\big(F(x(t-1))-F^*\big)\,,\label{eq:linconv1}
\end{aligned}
\end{equation}
\noindent with $\beta=\frac{\alpha m\varepsilon(2\pi\lambda^{2}-\varepsilon\Lambda^{2})}{\lambda\pi}$.\\
We next take expectation \sf{on both sides} of (\ref{eq:linconv1}) with respect to $\mathcal{F}_{t-2} \supseteq ...\supseteq \mathcal{F}_{0}$ recursively. Using the tower rule of expectations we have

\begin{equation}
\begin{aligned}
&\mathbb{E}\big[F(x(t))-F^*\textbf{\Big |}\mathcal{F}_{0}\big]=\mathbb{E}\big[F(x(t))-F^*\big]\leq\big(1-\beta\big)^{t}\big(F(x(0))-F^*\big)\,.\label{eq:lin}
\end{aligned}
\end{equation}
\sm{We now analyze the  sequence of $\left\{ x(t)\right\}$. By using the Taylor's theorem and the strong convexity of the objective function $F(\cdot)$, we have 
\[F(x(t))\geq F^*+\re{g(x^*)^\prime}(x(t)-x^*)+\frac{\alpha m}{2}\norm{x(t)-x^*}^2,\] where $\alpha m$ is the lower bound on the eigenvalues of $H(t)$ [c.f. Lemma \ref{lemma 1}].  We note that $g(x^*)=0$, therefore \[\norm{ x(t)-x^*}^2\leq\frac{2}{\alpha m}\big(F(x(t))-F^*\big).\] 
We next take expectations on both sides of the previous inequality and use Eq. (\ref{eq:lin}) to obtain
\[\mathbb{E}\Big[\norm{x(t)-x^*}^2\Big]\leq\frac{2}{\alpha m}\mathbb{E}\Big[F(x(t))-F^*\Big]\leq\frac{2\big(F(x(0))-F^*\big)}{\alpha m}(1-\beta)^t.\] Employing the Jensen's inequality for expectations yields \[\Bigg(\mathbb{E}\Big[\norm{x(t)-x^*}\Big]\Bigg)^2\leq\mathbb{E}\Big[\norm{\big(x(t)-x^*\big)}^2\Big]\leq\frac{2\big(F(x(0))-F^*\big)}{\alpha m}(1-\beta)^t.\] By taking square root on both sides of the previous relation, we obtain
\begin{equation}
\begin{aligned}
\mathbb{E}\Big[\norm{x(t)-x^*}\Big]\leq\Bigg(\frac{2\big(F(x(0))-F^*\big)}{\alpha m}\Bigg)^{1/2}\Big((1-\beta)^{1/2}\Big)^t. \label{eq:linX}
\end{aligned}
\end{equation}}
We note that Eq. (\ref{eq:lin}) and Eq. (\ref{eq:linX}) imply the global linear convergence in expectation only if $0<\beta<1$. We next argue that $0<\beta<1$. We note that if the stepsize parameter $\varepsilon$ satisfies the condition in Eq. (\ref{eq:step}), we have
\[2\pi\lambda^{2}-\varepsilon\Lambda^{2}>0,\] 
thus, $\beta>0$. We now show that $\beta<1$. We first rewrite $\beta$ as
\[\beta=\frac{2\alpha m\pi\varepsilon\lambda^{2}}{\lambda\pi}-\frac{\alpha m\varepsilon^2\Lambda^{2}}{\lambda\pi}.\] We note that $\frac{\alpha m\varepsilon^{2}\Lambda^{2}}{\lambda\pi}>0$ and $\lambda=\frac{1}{2(1-\delta)+\alpha M}$ [c.f. Lemma \ref{lemma:lemma 4.2}]. Therefore, \[\beta<\frac{2\alpha m\varepsilon\pi}{\big(2(1-\delta)+\alpha M\big)\pi}.\] 
Because $1-\delta>0$, we have $\alpha m<\alpha M+2(1-\delta)$, using this together with the fact that $\varepsilon<\frac{1}{2}$, we
obtain $\beta<1$.\end{proof}
\re{\begin{remark} The linear convergence rate depends on the constant $1-\beta$. The smaller $1-\beta$ is, the faster the algorithm converges. We note that the constant $\beta$ is increasing in the minimum activation probability, $\pi$, meaning that, smaller $\pi$ results in smaller $\beta$ and hence slower convergence. To illustrate this point , consider the problem with uniform activation probabilities, i.e., $p_i=\frac{1}{n}$, for all $i$. In this case, the constant $\beta$ is of order $\frac{1}{n}$, and increasing $n$ results in slower convergence.  
\end{remark}}
\subsection{Local Quadratic Rate of Convergence}\label{sec:Quad}
\sf{We now proceed to prove local quadratic convergence rate in expectation for our asynchronous network Newton algorithm. due to technical convenience, instead of $\norm{x(t)-x^*}$, we work with weighted error $\norm{D(t-1)^{1/2}\big(x(t)-x^*\big)}$ in our analysis.
In Lemmas \ref{reverseJensen} to \ref{recursion}, we prove some key relations that we use to establish an upper bound for the weighted error $\norm{D(t-1)^{1/2}\big(x(t)-x^*\big)}$ in Lemma\ref{linquad}. This upper bound is a summation of two terms, which are linear and quadratic \ss{functions} on the weighted error corresponding to the previous iterate. In Lemma \ref{linconvX}, we show that the \ss{weighted error} sequence $\big\{ \norm{D(t-1)^{1/2}\big(x(t)-x^*\big)}\big\}$ converges linearly in expectation. Finally, in Theorem \ref{localquad}, we prove that there exists an interval where $\big\{\norm{D(t-1)^{1/2}\big(x(t)-x^*\big)}\big\}$ sequence decreases with quadratic rate.}
\begin{lemma} \label{reverseJensen} Let $X$ be a non-negative random variable with $n$ different realizations $X_i$, each happens with probability $\pi\leq q_i\leq\Pi$. Then, 
\begin{equation*}
\mathbb{E}\big[X ^{2}\big]\leq \frac{1}{\pi}\big(\mathbb{E}\big[X \big]\big)^{2}\,.\label{eq:42}
\end{equation*}
\end{lemma}
\begin{proof}
Note that since $X_i$ is non-negative, we have $\sum_{i=1}^{n}X_{i}^{2}\leq\big(\sum_{i=1}^{n}X_{i}\big)^{2}$. Therefore,
\begin{equation*}
\begin{aligned}
&\mathbb{E}\big[X^{2}\big]=\sum_{i=1}^{n}q_i X_{i}^{2}=\sum_{i=1}^{n}\big(q_i^{1/2} X_{i}\big)^{2}\leq\big(\sum_{i=1}^{n}q_i^{1/2}X_{i}\big)^{2}=\big(\sum_{i=1}^{n}q_i^{-1/2}q_i X_{i}\big)^{2}\\& \leq\big(\pi^{-1/2}\sum_{i=1}^{n}q_i X_{i}\big)^{2}=\frac{1}{\pi}\big(\mathbb{E}\big[X\big]\big)^{2},\label{eq:43}
\end{aligned}
\end{equation*} where the last equality follows from the definition of expected value of a random variable.
\end{proof}
\begin{lemma} \label{error}
Consider the approximated Hessian inverse defined in Eq. (\ref{eq:defHatH}), \sm{then the following inequality holds for all $t\geq 0$}
\[D(t)^{1/2}(I-\hat{H}(t)^{-1}H(t))=\big(D(t)^{-1/2}BD(t)^{-1/2}\big)^2D(t)^{1/2}.\]
\end{lemma}
\begin{proof}
By using the definition of the Hessian matrix, $H(t)$, and its approximated inverse $\hat{H}(t)^{-1}$ from equations (\ref{eq:Hsplit}) and (\ref{eq:defHatH}), we have
\begin{align*}&I-\hat{H}(t)^{-1}H(t)=I-\Big(D(t)^{-1}+D(t)^{-1}BD(t)^{-1}\Big)\big(D(t)-B\big)\\&=I-\Big(I-D(t)^{-1}B+D(t)^{-1}B-\big(D(t)^{-1}B\big)^2\Big)=(D(t)^{-1}B)^2\\&=D(t)^{-1/2}\big(D(t)^{-1/2}BD(t)^{-1/2}\big)^2D(t)^{1/2}\end{align*}
By multiplying the previous relation by $D(t)^{1/2}$ \re{from the left,} we complete the proof.
\end{proof}
\sm{\begin{lemma}\label{lemma:expectnorm}
\sf{For all $t>0$}, consider matrices $\Phi(t)$, $D(t-1)$, and $B$ defined in Eq. (\ref{eq:15}) and (\ref{eq:defDB}) and recall the definition of $\rho$ from Lemma \ref{proposition 2}, then considering the history of the algorithm up to iteration $t$, if stepsize parameter $\varepsilon$ satisfies Eq. (\ref{eq:step}), then for any $y\in\mathbb{R}^n$ 
\begin{equation*}
\mathbb{E}\left[\norm{ \Big(I-\varepsilon
P^{-1}\Phi(t)+\varepsilon P^{-1}\Phi(t)Q(t-1)\Big)y}\textbf{\Big|}\mathcal{F}_{t-1}\right]\leq C_1 \norm {y}.\label{eq:expectbound}
\end{equation*}
where $Q(t-1)=\big(D(t-1)^{-1/2}BD(t-1)^{-1/2}\big)^2$, and $C_1=\Bigg(1+\varepsilon\max\Big\{\frac{\varepsilon}{\pi}-2, \frac{\varepsilon(1-\rho^2)^2}{\pi}-2(1-\rho^2)\Big\}\Bigg)^{1/2}<1$.
\end{lemma}}
\begin{proof}  
\sf{Note that if agent $i$ is active at iteration $t$, then the activation matrix realization is $\Phi^i$ and we have}
\begin{equation*}
\Big[\Big(I-\varepsilon P^{-1}\Phi^i+\varepsilon P^{-1}\Phi^iQ(t-1)\Big)y\Big]_j=\begin{cases}
\Big[\big(I-\varepsilon P^{-1}+\varepsilon P^{-1}Q(t-1)\big)y\Big]_i & \mbox{if $j=i$},\\
y_j & \mbox{otherwise}.
\end{cases}
\end{equation*}
Hence,
\begin{align*}
\norm{\Big(I-\varepsilon P^{-1}\Phi^i+\varepsilon P^{-1}\Phi^iQ(t-1)\Big)y}^2=\Big[\Big(I-\varepsilon P^{-1}+\varepsilon P^{-1}Q(t-1)\Big)y\Big]_i^2+\norm{
y_{-i}}^2,
\end{align*}where $y_{-i}$ is a vector with a zero at $i^{th}$ element and the rest of its elements are the same as vector $y$.
Taking expectation over all possible realizations of matrix $\Phi(t)$, we obtain
\begin{align*}
&\mathbb{E}\left[\norm{ \Big(I-\varepsilon P^{-1}\Phi(t)+\varepsilon P^{-1}\Phi(t)Q(t-1)\Big)y}^2\textbf{\Big|}\mathcal{F}_{t-1}\right]\\&=\sum_{i=1}^n p_i\norm{\Big(I-\varepsilon P^{-1}\Phi^i+\varepsilon P^{-1}\Phi^iQ(t-1)\Big)y}^2\\&=\sum_{i=1}^n p_i\Big(\Big[\Big(I-\varepsilon P^{-1} +\varepsilon P^{-1}Q(t-1)\Big)y\Big]_i^2+\norm{
y_{-i}}^2\Big)\\&=\sum_{i=1}^n p_i\Big[\Big(I-\varepsilon P^{-1} +\varepsilon P^{-1}Q(t-1)\Big)y\Big]_i^2+\sum_{i=1}^n(1-p_i) y_i^2\\&= y^{\prime}\Big(I-\varepsilon P^{-1} +\varepsilon P^{-1}Q(t-1)\Big)^\prime P \Big(I-\varepsilon P^{-1} +\varepsilon P^{-1}Q(t-1)\Big)y+y^\prime (I-P)y,
\end{align*}
where the last equality comes from the fact that for any $z\in\mathbb{R}^n$ we have $z^\prime P z=\sum_{i=1}^n p_i z_i^2$. By some algebraic manipulations, we obtain
\begin{equation}
\begin{aligned}&\mathbb{E}\left[\norm{ \Big(I-\varepsilon P^{-1}\Phi(t)+\varepsilon P^{-1}\Phi(t)Q(t-1)\Big)y}^2\textbf{\Big|}\mathcal{F}_{t-1}\right]=\\&y^\prime y+y^\prime\Big(-2\varepsilon\big(I-Q(t-1)\big)+\varepsilon^2\big(I-Q(t-1)\big)P^{-1}\big(I-Q(t-1)\big)\Big)y.\label{eq:NormEXP}\end{aligned}\end{equation}  
We note that matrix $\big(I-Q(t-1)\big)$ is symmetric and $P^{-1}$ is a positive definite matrix. Therefore, for every $y\in\mathbb{R}^n $ we have \[y^\prime\big(I-Q(t-1)\big)P^{-1}\big(I-Q(t-1)\big)y\leq\frac{1}{\pi}y^\prime\big(I-Q(t-1)\big)\big(I-Q(t-1)\big)y,\] where $\frac{1}{\pi}$ is the largest eigenvalue of $P^{-1}$. Therefore, we can bound the second term in the right hand side of Eq. (\ref{eq:NormEXP}) as follows \begin{equation}\begin{aligned} &y^\prime\big(-2\varepsilon\big(I-Q(t-1)\big)+\varepsilon^2\big(I-Q(t-1)\big)P^{-1}\big(I-Q(t-1)\big)\big)y\\&\leq -2\varepsilon y^\prime\big(I-Q(t-1)\big)y+\frac{\varepsilon^2}{\pi}y^\prime\big(I-Q(t-1)\big)^2y.\label{eq:Pbound}\end{aligned}\end{equation}
We note that $\big(I-Q(t-1)\big)$ is a symmetric matrix and can be diagonalized as $\big(I-Q(t-1)\big)=VUV^\prime$, where $V\in\mathbb{R}^{n\times n}$ is an orthonormal matrix, i.e., $VV^\prime=I$, whose $i^{th}$ column $v_i$ is the eigenvector of $\big(I-Q(t-1)\big)$ and $v_i'v_j=0$ and $U$ is the diagonal matrix whose diagonal elements, $\mu_i$, are the corresponding eigenvalues. We also note that since $V$ is an orthonormal matrix, $\big(I-Q(t-1)\big)^2=VU^2V^\prime$. Using this diagonalization, we have
\begin{equation}\begin{aligned} &-2\varepsilon y^\prime\big(I-Q(t-1)\big)y+\frac{\varepsilon^2}{\pi}y^\prime\big(I-Q(t-1)\big)^2y=-2\varepsilon y^\prime VUV^\prime y+\frac{\varepsilon^2}{\pi}y^\prime VU^2V^\prime y=\\&-2\varepsilon\sum_{i=1}^n \mu_i(v_i^\prime y)^2+\frac{\varepsilon^2}{\pi}\sum_{i=1}^n \mu_i^2(v_i^\prime y)^2=\varepsilon\sum_{i=1}^n \big(\frac{\varepsilon\mu_i^2}{\pi}-2\mu_i\big)(v_i^\prime y
)^2\leq\varepsilon\max_{\mu_i}\big(\frac{\varepsilon\mu_i^2}{\pi}-2\mu_i\big)\sum_{i=1}^n(v_i^\prime y
)^2. \label{eq:diagonalization} \end{aligned}\end{equation} We note that $\frac{\varepsilon\mu_i^2}{\pi}-2\mu_i$ is a convex function in $\mu_i$ which reaches its minimum value at $\mu_i=\frac{\pi}{\varepsilon}>\frac{1}{2}$.
Considering the definition of $Q(t-1)$ and using the result of Lemma \ref{proposition 2} to bound its eigenvalues we have for all $t$\[(1-\rho^2)I\preceq I-Q(t)\preceq I,\] hence, $0<1-\rho^2\leq\mu_i\leq 1.$  Therefore, the maximum value of $\frac{\varepsilon\mu_i^2}{\pi}-2\mu_i$ happens at either $\mu_i=1-\rho^2$ or $\mu_i=1$, i.e., \begin{equation}\begin{aligned}\max_{\mu_i\in[1-\rho^2, 1]}\big(\frac{\varepsilon\mu_i^2}{\pi}-2\mu_i\big)= \max\Big\{\frac{\varepsilon}{\pi}-2, \frac{\varepsilon(1-\rho^2)^2}{\pi}-2(1-\rho^2)\Big\}. \label{eq:max}\end{aligned}\end{equation} 
Combining Eq. (\ref{eq:Pbound}), Eq. (\ref{eq:diagonalization}), and Eq. (\ref{eq:max}) and the fact that $\sum_{i=1}^n(v_i^\prime y)=y^\prime V V^\prime y=y^\prime y$, we obtain 
\begin{equation}\begin{aligned} &y^\prime\big(-2\varepsilon\big(I-Q(t-1)\big)+\varepsilon^2\big(I-Q(t-1)\big)P^{-1}\big(I-Q(t-1)\big)\big)y\\&\leq \max\Big\{\frac{\varepsilon}{\pi}-2, \frac{\varepsilon(1-\rho^2)^2}{\pi}-2(1-\rho^2)\Big\}\norm{y}^2.\label{eq:maxbound}\end{aligned}\end{equation}
We now combine Eq. (\ref{eq:NormEXP}) and Eq. (\ref{eq:maxbound}) and use Jensen's inequality to obtain
\begin{equation*}
\begin{aligned}&\Bigg(\mathbb{E}\Big[\norm{ \Big(I-\varepsilon
P^{-1}\Phi(t)+\varepsilon P^{-1}\Phi(t)Q(t-1)\Big)y}\textbf{\Big|}\mathcal{F}_{t-1}\Big]\Bigg)^2  \\&\leq\mathbb{E}\left[\norm{ \Big(I-\varepsilon P^{-1}\Phi(t)+\varepsilon P^{-1}\Phi(t)Q(t-1)\Big)y}^2\textbf{\Big|}\mathcal{F}_{t-1}\right]\\&\leq\Bigg(1+\varepsilon\max\Big\{\frac{\varepsilon}{\pi}-2, \frac{\varepsilon(1-\rho^2)^2}{\pi}-2(1-\rho^2)\Big\}\Bigg)\norm{y}^2
\end{aligned}\end{equation*}
We emphasize that since the stepsize parameter $\varepsilon$ satisfies Eq. (\ref{eq:step}), we have $\frac{\varepsilon}{\pi}<2$, therefore \[\max\Big\{\frac{\varepsilon}{\pi}-2, \frac{\varepsilon(1-\rho^2)^2}{\pi}-2(1-\rho^2)\Big\}<0.\]
Hence,
\[C_1=\Bigg(1+\varepsilon\max\Big\{\frac{\varepsilon}{\pi}-2, \frac{\varepsilon(1-\rho^2)^2}{\pi}-2(1-\rho^2)\Big\}\Bigg)^{1/2}<1\]
\end{proof}
\begin{lemma} \label{recursion} Consider the asynchronous network Newton algorithm as in Algorithm \ref{async NN}, and remember the definition of $D(t-1)$ and $B$ from Eq. (\ref{eq:defDB}), then for any $y\in\mathbb{R}^n$ we have
\[\norm {D(t-1)^{1/2}y}\leq\Big(1+C_2\norm{ g(t-2)}^{1/2}\Big)\norm{ D(t-2)^{1/2}y},\]
where $C_2=\Big(\frac{\varepsilon\alpha L \Lambda}{\pi\big(2(1-\Delta)+\alpha m\big)}\Big)^{1/2}$.
\end{lemma}
\begin{proof}
We note that if $\norm{ D(t-1)^{1/2}y}\leq \norm {D(t-2)^{1/2}y}$, the claim is true because $C_2>0$. Therefore, we consider the case with $\norm{ D(t-1)^{1/2}y}> \norm {D(t-2)^{1/2}y}$. We next use the Lipschitz property of the Hessian \sm{ [c.f. Lemma \ref{lemma 01}],} to obtain 
\begin{equation*}\norm{ D(t-1)-D(t-2)}=\norm {H(t-1)-H(t-2)}\leq\alpha L\norm{ x(t-1)-x(t-2)}.  \label{eq:recursion}\end{equation*}
where $\alpha L$ is the Lipschitz constant of the Hessian matrix according to the result of Lemma \ref{lemma 01}. We also note that \[\Big\vert y^\prime D(t-1)y - y^\prime D(t-2)y \Big\vert=\big\vert y^{\prime} \big(D(t-1)-D(t-2)\big)y\big\vert\leq\alpha L\norm{ x(t-1)-x(t-2)}\norm{ y}^2.\]
Note that $y^\prime D(t-1)y =\norm{ D(t-1)^{1/2}y}^2$ and  $y^\prime D(t-2)y =\norm {D(t-2)^{1/2}y}^2$, then using triangular inequality together with the fact that $\norm{ D(t-1)^{1/2}y}> \norm{D(t-2)^{1/2}y}$, we have
\begin{equation*}\norm{ D(t-1)^{1/2}y}^2\leq\norm{ D(t-2)^{1/2}y }^2+\alpha L\norm{ x(t-1)-x(t-2)}\norm {y}^2. \end{equation*}
\sm{We note that for every $a,b,c\in\mathbb{R}$, if $a^2\leq b^2+c^2$ then we have $\vert a\vert\leq\vert b\vert+\vert c\vert$ . Therefore, }
\begin{equation}\norm{ D(t-1)^{1/2}y}\leq\norm {D(t-2)^{1/2}y }+\Big(\alpha L\norm{ x(t-1)-x(t-2)}\Big)^{1/2}\norm{ y}.\label{recurbound} \end{equation}
In this step, we find an upper bound for $\norm{y}$ in terms of $\norm{ D(t-2)^{1/2}y}$. We note that \begin{align*}\mu_{min}\big(D(t-2)^{1/2}\big)\norm {y}\leq\norm{ D(t-2)^{1/2}y},\end{align*} where $\mu_{min}\big(D(t-2)^{1/2}\big)$ is the minimum eigenvalue of the positive definite matrix $D(t-2)^{1/2}$. Hence, \sm{using the result of Lemma \ref{lemma 1}, we have}
\begin{align}\norm{ y}\leq \frac{1}{\mu_{min}\big(D(t-2)^{1/2}\big)}\norm{ D(t-2)^{1/2}y}\leq \frac{1}{\big(2(1-\Delta)+\alpha m\big)^{1/2}}\norm{ D(t-2)^{1/2}y}.\label{boundNorm}\end{align}
We next combine relations (\ref{recurbound}) and (\ref{boundNorm}) to obtain 
\begin{align*}
\norm{D(t-1)^{1/2}y}\leq\Bigg(1+\Bigg(\frac{\alpha L\norm{ x(t-1)-x(t-2)}}{2(1-\Delta)+\alpha m}\Bigg)^{1/2}\Bigg)\norm{ D(t-2)^{1/2}y}
\end{align*}
Finally, we use the asynchronous network Newton iteration defined in Eq. (\ref{eq:xUpdateAsync}) to substitute $x(t-1)-x(t-2)=-\varepsilon P^{-1}\Phi(t-1)\hat{H}(t-2)^{-1}g(t-2)$ to get \begin{align*}&\norm {D(t-1)^{1/2}y}\leq\\&\Bigg(1+\Big(\frac{\varepsilon\alpha L}{2(1-\Delta)+\alpha m}\Big)^{1/2}\norm{ P^{-1}\Phi(t-1)\hat{H}(t-2)^{-1}g(t-2)}^{1/2}\Bigg)\norm{ D(t-2)^{-1}y}.\end{align*}
By using Cauchy-Schwarz inequality and the facts that $\norm{\hat{H}(t-2)^{-1}}\leq\Lambda$ [c.f. Lemma \ref{lemma:lemma 4.2}], $\norm{P^{-1}}\leq\frac{1}{\pi}$ and $\norm{\Phi(t-1)}=1$ for all realizations, we complete the proof.
\end{proof}
\sm{\begin{lemma} \label{linquad} Consider the asynchronous network Newton algorithm as in Algorithm \ref{async NN} with stepsize parameter $\varepsilon$ that satisfies Eq. (\ref{eq:step}), and recall the definition of $0<\beta<1$ from Theorem \ref{linconv} , $\rho<1$ from Lemma \ref{proposition 2}, $\lambda$ and $Lambda$ from Lemma \ref{lemma:lemma 4.2},  $C_1<1$ from Lemma \ref{lemma:expectnorm}, and $C_2$ from Lemma \ref{recursion}, then the sequence the \sf{weighted} errors $\left\{ \norm{ D(t-1)^{1/2}\big(x(t)-x^*\big)}\right\}$ satisfies
\begin{equation*}
\begin{aligned}
&\mathbb{E}\Big[\norm{ D(t-1)^{1/2}\big(x(t)-x^*\big)}\Big]\leq \Gamma_1\Bigg(\mathbb{E}\Big[\norm{ D(t-2)^{1/2}\big(x(t-1)-x^*\big)}\Big]\Bigg)^2\\&+\Gamma(t)\mathbb{E}\Big[\norm{D(t-2)^{1/2}\big(x(t-1)-x^*\big)}\Big], 
\end{aligned}
\end{equation*}
\noindent where $\Gamma_{1}=\frac{\big(2(1-\delta)+\alpha M\big)^{1/2}\alpha L\varepsilon \Lambda}{2\pi^2\big(2(1-\Delta)+\alpha m\big)}$ and $\Gamma(t)=C_1\Big(1+C_3(1-\beta)^{\frac{t-2}{4}}\Big)$ with $C_3=C_2 \Big(\frac{2}{\lambda\pi^2}\big(F(x(0))-F^*\big)\Big)^{1/4}.$
\end{lemma}}
\begin{proof} 
By adding and subtracting $x(t-1)$ and $\varepsilon P^{-1}\Phi(t)\hat{H}(t-1)^{-1}H(t-1)\left(x(t-1)-x^*\right)$ from $x(t)-x^*$ we have
\begin{equation*}
\begin{aligned}
&x(t)-x^*=x(t)-x(t-1)+\varepsilon P^{-1}\Phi(t)\hat{H}(t-1)^{-1}H(t-1)\left(x(t-1)-x^*\right)+x(t-1)-x^*\\&-\varepsilon P^{-1}\Phi(t)\hat{H}(t-1)^{-1}H(t-1)\left(x(t-1)-x^*\right).
\end{aligned}
\end{equation*}
We next substitute $x(t)-x(t-1)=-\varepsilon P^{-1}\Phi(t)\hat{H}(t-1)^{-1}g(t-1)$ using Eq. (\ref{eq:xUpdateAsync}) and add and subtract $\varepsilon P^{-1}\Phi(t)\big(x(t-1)-x^*\big)$ in the above relation to obtain
\begin{align*} &x(t)-x^*=\varepsilon P^{-1}\Phi(t)\hat{H}(t-1)^{-1}\Big(H(t-1)\big(x(t-1)-x^*\big)-g(t-1)\Big)\\&+\big(I-\varepsilon P^{-1}\Phi(t)\big)\big(x(t-1)-x^*\big)+\varepsilon P^{-1}\Phi(t)\big(I-\hat{H}(t-1)^{-1}H(t-1)\big)\big(x(t-1)-x^*\big).\end{align*}
By multiplying both sides of the previous equality by the diagonal matrix $D(t-1)^{1/2}$ and using the result of Lemma \ref{error} that  $D(t-1)^{1/2}\big(I-\hat{H}(t-1)^{-1}H(t-1)\big)=\big(D(t-1)^{-1/2}BD(t-1)^{-1/2}\big)^2D(t-1)^{1/2}$, we have
\begin{align*}&D(t-1)^{1/2}\big(x(t)-x^*\big)=\varepsilon P^{-1} D(t-1)^{1/2}\Phi(t)\hat{H}(t-1)^{-1}\Big(H(t-1)\big(x(t-1)-x^*\big)-g(t-1)\Big)\\&+\big(I-\varepsilon P^{-1}\Phi(t)\big)D(t-1)^{1/2}\big(x(t-1)-x^*\big)\\&+\varepsilon P^{-1}\Phi(t)\big(D(t-1)^{-1/2}BD(t-1)^{-1/2}\big)^2D(t-1)^{1/2}\big(x(t-1)-x^*\big),\end{align*}
\sf{where we used the commutative property of the multiplication of diagonal matrices $D(t-1)^{1/2} \,, \,P\,, \,\Phi(t)\,,$ and $(I-\varepsilon P^{-1}\Phi(t))$.} We then take norms on both sides and use triangular and Cauchy-Schwarz inequalities to obtain
\begin{equation}
\begin{aligned}
&\norm{ D(t-1)^{1/2}\big(x(t)-x^*\big)}\leq\\&\varepsilon\norm{ P^{-1}D(t-1)^{1/2}\Phi(t)\hat{H}(t-1)^{-1}}\norm{ H(t-1)\big(x(t-1)-x^*\big)-g(t-1)}\\&+\norm{\Big( I-\varepsilon P^{-1}\Phi(t)+\varepsilon P^{-1}\Phi(t)\big(D(t-1)^{-1/2}BD(t-1)^{-1/2}\big)^2\Big) D(t-1)^{1/2}\big(x(t-1)-x^*\big)}. \label{eq:xNorm}
\end{aligned}
\end{equation}
\fm{ We next find an upper  bound, in terms of $\norm{ D(t-2)^{1/2}\big(x(t-1)-x^*\big)}$, for the first term of the summation in the right hand side of Eq. (\ref{eq:xNorm}). 
Applying the result of Lemma \ref{lemma:lemma 4.8} with $v=x^*$ and $u=x(t-1)$ and considering the fact that $\nabla F(x^*)=0$, yield
\begin{align}&\norm{ H(t-1)\big(x(t-1)-x^*\big)-g(t-1)}\leq\frac{\alpha L}{2}\norm{ x(t-1)-x^*}^2, \label{eq:intbound}\end{align} where $\alpha L$ is the Lipschitz constant of the Hessian matrix according to the result of Lemma \ref{lemma 01}. 
Using the definition of $\mu_{min}(\cdot)$ we have \[\norm{ x(t-1)-x^*}\leq\frac{1}{\mu_{min}\big(D(t-2)^{1/2}\big)}\norm{ D(t-2)^{1/2}\big(x(t-1)-x^*\big)},\] \sf{hence, by using Lemma \ref{lemma 1} to bound $\mu_{min}\big(D(t-2)^{1/2}\big)$ we have }\begin{equation}\begin{aligned}\norm{ x(t-1)-x^*}^2\leq\frac{1}{2(1-\Delta)+\alpha m}\norm{ D(t-2)^{1/2}\big(x(t-1)-x^*\big)}^2.\label{eq:boundx}\end{aligned}\end{equation}
We next combine Eq. (\ref{eq:intbound}) and Eq. (\ref{eq:boundx}) and use the upper bounds $\norm{P^{-1}}\leq\frac{1}{\pi}$, $\norm{ D(t-1)^{1/2}}\leq\big(2(1-\delta)+\alpha M\big)^{1/2}$, and $\norm{\hat{H}(t-1)^{-1}}\leq\Lambda$ [c.f. Assumption \ref{ass:activeprob}, Lemma \ref{lemma 1}, and Lemma \ref{lemma:lemma 4.2}], together with the fact that \sf{for all realizations of the stochastic activation matrix,} $\norm{\Phi(t)}=1$, and obtain
\begin{equation}\begin{aligned}&\varepsilon\norm{ P^{-1}D(t-1)^{1/2}\Phi(t)\hat{H}(t-1)^{-1}}\norm{ H(t-1)\big(x(t-1)-x^*\big)-g(t-1)}\\&\leq\frac{\big(2(1-\delta)+\alpha M\big)^{1/2}\alpha L\varepsilon \Lambda}{2\pi\big(2(1-\Delta)+\alpha m\big)}\norm{ D(t-2)^{1/2}\big(x(t-1)-x^*\big)}^2.\label{eq:firstTerm}\end{aligned}\end{equation}
We now substitute Eq. (\ref{eq:firstTerm}) in Eq. (\ref{eq:xNorm}) to obtain
\begin{equation*}
\begin{aligned}
&\norm{D(t-1)^{1/2}\big(x(t)-x^*\big)}\leq\frac{\big(2(1-\delta)+\alpha M\big)^{1/2}\alpha L\varepsilon \Lambda}{2\pi\big(2(1-\Delta)+\alpha m\big)}\norm{ D(t-2)^{1/2}\big(x(t-1)-x^*\big)}^2+\\&\norm{\Big( I-\varepsilon P^{-1}\Phi(t)+\varepsilon P^{-1}\Phi(t)\big(D(t-1)^{-1/2}BD(t-1)^{-1/2}\big)^2\Big) D(t-1)^{1/2}\big(x(t-1)-x^*\big)}.
\end{aligned}
\end{equation*}
This inequality holds for any random activation of the agents. We now note that conditioned on $\mathcal{F}_{t-1}$, matrix $\Phi(t)$  and $x(t)$ are random and $x(t-1)$ is deterministic, we can hence take expectation on both sides of the above inequality and have
\begin{equation}
\begin{aligned}
&\mathbb{E}\Big[\norm{D(t-1)^{1/2}\big(x(t)-x^*\big)}\textbf{\Big|}\mathcal{F}_{t-1}\Big]\leq\\&\frac{\big(2(1-\delta)+\alpha M\big)^{1/2}\alpha L\varepsilon \Lambda}{2\pi\big(2(1-\Delta)+\alpha m\big)}\norm{ D(t-2)^{1/2}\big(x(t-1)-x^*\big)}^2+\mathbb{E}\Big[\Big\Vert\Big( I-\varepsilon P^{-1}\Phi(t)\\&+\varepsilon P^{-1}\Phi(t)\big(D(t-1)^{-1/2}BD(t-1)^{-1/2}\big)^2\Big) D(t-1)^{1/2}\big(x(t-1)-x^*\big)\Big\Vert\textbf{\Big|}\mathcal{F}_{t-1}\Big].\label{eq:afterE}
\end{aligned}
\end{equation}
We then consider the second term of the summation in the right hand side of Eq. (\ref{eq:afterE}). Using the result of Lemma \ref{lemma:expectnorm} with $y=D(t-1)^{1/2}\big(x(t-1)-x^*\big)$ we have
\begin{equation}\begin{aligned}&\mathbb{E}\Big[\Big\Vert\Big( I-\varepsilon P^{-1}\Phi(t)+\varepsilon P^{-1}\Phi(t)\big(D(t-1)^{-1/2}BD(t-1)^{-1/2}\big)^2\Big)\times\\& D(t-1)^{1/2}\big(x(t-1)-x^*\big)\Big\Vert\textbf{\Big|}\mathcal{F}_{t-1}\Big]\leq C_1 \norm{D(t-1)^{1/2}\big(x(t-1)-x^*\big)}\\&\leq C_1\Big(1+C_2\norm{ g(t-2)}^{1/2}\Big)\norm{ D(t-2)^{1/2}\big(x(t-1)-x^*\big)}, \label{eq:secondTerm}\end{aligned}\end{equation} where in the second inequality we use Lemma \ref{recursion} to bound $\norm{ D(t-1)^{1/2}\big(x(t-1)-x^*\big)}$ in terms of $\norm{ D(t-2)^{1/2}\big(x(t-1)-x^*\big)}$.
We now substitute Eq. (\ref{eq:secondTerm}) in Eq. (\ref{eq:afterE}) to obtain
\begin{equation}
\begin{aligned}
&\mathbb{E}\Big[\norm{ D(t-1)^{1/2}\big(x(t)-x^*\big)}\textbf{\Big|}\mathcal{F}_{t-1}\Big]\leq\\&\frac{\big(2(1-\delta)+\alpha M\big)^{1/2}\alpha L\varepsilon \Lambda}{2\pi\big(2(1-\Delta)+\alpha m\big)}\norm{ D(t-2)^{1/2}\big(x(t-1)-x^*\big)}^2\\&+ C_1\Big(1+C_2\norm{ g(t-2)}^{1/2}\Big)\norm{ D(t-2)^{1/2}\big(x(t-1)-x^*\big)}. \label{eq:condexp}
\end{aligned}
\end{equation}
We next take expectations on both sides of Eq. (\ref{eq:condexp}) and have
\begin{equation}
\begin{aligned}
&\mathbb{E}\Big[\norm{ D(t-1)^{1/2}\big(x(t)-x^*\big)}\Big]\leq\\&\frac{\big(2(1-\delta)+\alpha M\big)^{1/2}\alpha L\varepsilon \Lambda}{2\pi\big(2(1-\Delta)+\alpha m\big)}\mathbb{E}\Big[\norm{ D(t-2)^{1/2}\big(x(t-1)-x^*\big)}^2\Big] \\&+C_1\mathbb{E}\Big[\Big(1+C_2\norm{ g(t-2)}^{1/2}\Big)\norm{ D(t-2)^{1/2}\big(x(t-1)-x^*\big)}\Big]. \label{expineq}
\end{aligned}
\end{equation}
We now focus on the second expected value in the right hand side of Eq. (\ref{expineq}). We have 
\begin{equation}\begin{aligned} &\mathbb{E}\Big[\Big(1+C_2\norm{g(t-2)}^{1/2}\Big)\norm{ D(t-2)^{1/2}\big(x(t-1)-x^*\big)}\Big]=\\&\mathbb{E}\Big[\norm{ D(t-2)^{1/2}\big(x(t-1)-x^*\big)}\Big]+  C_2  \mathbb{E}\Big[\norm{g(t-2)}^{1/2}\norm{ D(t-2)^{1/2}\big(x(t-1)-x^*\big)}\Big].\label{eq:Eproduct}\end{aligned}\end{equation}
We next study the second term in Eq. (\ref{eq:Eproduct}). We note that the Cauchy-Schwarz inequality in the context of the expectation states that for any two random variables $X$ and $Y$ such that $\mathbb{E}[X]$, $\mathbb{E}[Y]$, and $\mathbb{E}[XY]$ exist, we have 
\[\Big(\mathbb{E}[XY]\Big)^2\leq\mathbb{E}\big[X^2\big]\mathbb{E}\big[Y^2\big],\] \sf{hence, if $X,Y\geq 0$ we have  \[\mathbb{E}[X^{1/2}Y]\leq\Big(\mathbb{E}\big[X\big]\mathbb{E}\big[Y^2\big]\Big)^{1/2}.\]}
Therefore,
\begin{equation}\begin{aligned}&\mathbb{E}\Big[\norm{ g(t-2)}^{1/2}\norm{ D(t-2)^{1/2}\big(x(t-1)-x^*\big)}\Big]\\&\leq\Bigg(\mathbb{E}\big[\norm{ g(t-2)}\big]\Bigg)^{1/2}\Bigg(\mathbb{E}\Big[\norm{ D(t-2)^{1/2}\big(x(t-1)-x^*\big)}^2\Big]\Bigg)^{1/2}.\label{eq:productbound}\end{aligned}
\end{equation}
We next use the result of Lemma \ref{functionbound} on the properties of strongly convex functions together with the fact that $H(t)\preceq \big(2(1-\delta)+\alpha M\big)I=\frac{1}{\lambda}I$, to find an upper bound for $\mathbb{E}\Big[\norm{ g(t-2)}\Big]$ as follows
\[\norm{ g(t-2)}\leq\Big(\frac{2}{\lambda}\Big(F\big(x(t-2)\big)-F^{*}\Big)\Big)^{1/2}.\]
By taking expectation \sf{on both sides} of the above inequality and using the Jensen's inequality for concave functions together with the linear convergence result from Theorem \ref{linconv}, we have
\begin{equation}\begin{aligned}&\mathbb{E}\big[\norm{ g(t-2)}\big]
\leq\mathbb{E}\Big[\Big(\frac{2}{\lambda}\Big(F\big(x(t-2)\big)-F^{*}\Big)\Big)^{1/2}\Big]\leq\Big(\frac{2}{\lambda}\Big)^{1/2}\Big(\mathbb{E}\big[F\big(x(t-2)\big)-F^{*}\big]\Big)^{1/2}\\&\leq\Big(\frac{2}{\lambda}\Big)^{1/2}(1-\beta)^{\frac{t-2}{2}}\big(F\big(x(0)\big)-F^{*}\big)^{1/2}. \label{eq:boundgNorm}\end{aligned}\end{equation}
We also note that considering the result of Lemma \ref{reverseJensen} we have
\begin{equation}\begin{aligned}\mathbb{E}\Big[\norm{ D(t-2)^{1/2}\big(x(t-1)-x^*\big)}^2\Big]\leq \frac{1}{\pi}\Bigg(\mathbb{E}\Big[\norm{ D(t-2)^{1/2}\big(x(t-1)-x^*\big)}\Big]\Bigg)^{2}. \label{boundxNorm}\end{aligned}\end{equation}
By substituting Eq. (\ref{eq:boundgNorm}) and Eq. (\ref{boundxNorm}) in Eq. (\ref{eq:productbound}) and combining the result with Eq. (\ref{eq:Eproduct}), we obtain \begin{equation}\begin{aligned}&\mathbb{E}\Big[\Big(1+C_2\norm{g(t-2)}^{1/2}\Big)\norm{ D(t-2)^{1/2}\big(x(t-1)-x^*\big)}\Big]\\&\leq \Bigg(1+C_2\Big(\frac{2\big(F(x(0))-F^*\big)}{\lambda\pi^2}(1-\beta)^{t-2}\Big)^{1/4}\Bigg)\mathbb{E}\Big[\norm{ D(t-2)^{1/2}\big(x(t-1)-x^*\big)}\Big].\label{lastbound}\end{aligned}\end{equation}
Finally, we combine  Eq. (\ref{lastbound}), Eq. (\ref{boundxNorm}) and Eq. (\ref{expineq}) to obtain
\begin{equation*}
\begin{aligned}
&\mathbb{E}\Big[\norm{ D(t-1)^{1/2}\big(x(t)-x^*\big)}\Big]\leq\\& \frac{\big(2(1-\delta)+\alpha M\big)^{1/2}\alpha L\varepsilon \Lambda}{2\pi^2\big(2(1-\Delta)+\alpha m\big)}\Bigg(\mathbb{E}\Big[\norm{ D(t-2)^{1/2}\big(x(t-1)-x^*\big)}\Big]\Bigg)^2\\&+ C_1\Bigg(1+C_2\Big(\frac{2\big(F(x(0))-F^*\big)}{\lambda\pi^2}(1-\beta)^{t-2}\Big)^{1/4}\Bigg)\mathbb{E}\Big[\norm{ D(t-2)^{1/2}\big(x(t-1)-x^*\big)}\Big]. \label{expineqnew} 
\end{aligned}
\end{equation*}}
\end{proof}
\begin{lemma}\label{linconvX}
Consider the asynchronous network Newton iterate as in Algorithm \ref{async NN},   if the stepsize parameter $\varepsilon$ satisfies Eq. (\ref{eq:step}), then the sequence $\big\{\norm{ D(t-1)^{1/2} \big(x(t)-x^*\big)}\}$ converges linearly in expectation.
\end{lemma}
\begin{proof}
By using the Taylor's theorem and the strong convexity of the objective function $F(x(t))$, we have 
\[F(x(t))\geq F^*+g(x^*)(x(t)-x^*)+\frac{\alpha m}{2}\norm{x(t)-x^*}^2,\]where $\alpha m$ is the lower bound on the eigenvalues of $H(t)$ [c.f. Lemma \ref{lemma 1}].  We note that $g(x^*)=0$, therefore \[\norm{ x(t)-x^*}^2\leq\frac{2}{\alpha m}\big(F(x(t))-F^*)\big).\] By multiplying both sides by $\norm{ D(t-1)^{1/2}}$ and using the Cauchy-Schwarz inequality and \sf{Lemma \ref{lemma 1} to bound $\norm{ D(t-1)^{1/2}}$,} we obtain \[\norm{ D(t-1)^{1/2}\big(x(t)-x^*\big)}^2\leq\norm{ D(t-1)^{1/2}}^2\norm{ x(t)-x^*}^2\leq\frac{2\big(2(1-\delta)+\alpha M \big)}{\alpha m}\big(F(x(t))-F^*\big).\]
We next take expectation \sf{on both sides} of the previous inequality and apply the result of Lemma \ref{linconv} to obtain
\begin{align*}&\mathbb{E}\Big[\norm{ D(t-1)^{1/2}\big(x(t)-x^*\big)}^2\Big]\leq\frac{2\big(2(1-\delta)+\alpha M \big)}{\alpha m}\mathbb{E}\Big[\big(F(x(t))-F^*)\big)\Big]\\&\leq\frac{2\big(2(1-\delta)+\alpha M \big)\big(F(x(0))-F^*\big)}{\alpha m}(1-\beta)^t.\end{align*} Employing the Jensen's inequality for expectations yields \begin{align*}&\Bigg(\mathbb{E}\Big[\norm{ D(t-1)^{1/2}\big(x(t)-x^*\big)}\Big]\Bigg)^2\leq\mathbb{E}\Big[\norm{ D(t-1)^{1/2}\big(x(t)-x^*\big)}^2\Big]\\&\leq\frac{2\big(2(1-\delta)+\alpha M \big)\big(F(x(0))-F^*\big)}{\alpha m}(1-\beta)^t.\end{align*} By taking square root \sf{on both sides} of the previous relation we complete the proof. 

\end{proof}
\begin{theorem}\label{localquad}
Consider the asynchronous network Newton iterate as in Algorithm \ref{async NN} and recall the definition of $\Gamma_1$ and $\Gamma(t)$ and $C_3$ from Lemma \ref{linquad}, then for all $t$ with
 \begin{align}t> \frac{4\ln{\frac{1-C_1}{C_3C_1}}}{\ln{(1-\beta)}}+2, \label{tbound}\end{align}
there exists  $0<\theta<\frac{1-\Gamma(t)}{\Gamma_1\Gamma(t)}$, such that the sequence $\mathbb{E}\left[\big\Vert D(t-1)^{1/2}\big(x(t)-x^*\big)\big\Vert\right]$ satisfies
\begin{equation}
\theta\Gamma(t)\leq\mathbb{E}\left[\big\Vert D(t-1)^{1/2}\big(x(t)-x^*\big)\big\Vert\right]<\frac{\theta}{\theta\Gamma_1+1}, \label{eq:condition}
\end{equation}
and decreases with a quadratic rate in expectation in this interval.
\end{theorem}
\begin{proof}
We note that $\frac{1-\Gamma(t)}{\Gamma_1\Gamma(t)}>0$ if and only if $\Gamma(t)<1$. Next, we show that  for all $t$ satisfying Eq. (\ref{tbound}), $\Gamma(t)<1$. \\Recall the definition of $\Gamma(t)$ from Lemma \ref{linquad} , we have
\[\Gamma(t)=C_1\Big(1+C_3(1-\beta)^{\frac{t-2}{4}}\Big).\]
To have $\Gamma(t)<1$, we need
\[1+C_3(1-\beta)^{\frac{t-2}{4}}<\frac{1}{C_1},\] therefore, \[(1-\beta)^{\frac{t-2}{4}}<\frac{1-C_1}{C_3C_1}.\]
Taking logarithm to the base $1-\beta<1$ \sf{of both sides} of the above inequality flips the direction of the inequality and results in a lower bound for $t$ as \[t>4\log_{1-\beta}^{\frac{1-C_1}{C_3C_1}}+2,\] which is equal to the lower bound in Eq. (\ref{tbound}) by changing the base of the logarithm. \sf{Therefore, using the fact that $\left\{\big\Vert D(t-1)^{1/2}\big(x(t)-x^*\big)\big\Vert\right\}$ decreases linearly in expectation [c.f. Lemma \ref{linconvX}], for all iterations $t$ satisfying Eq. (\ref{tbound})  there exists $\theta$ such that $\left\{\big\Vert D(t-1)^{1/2}\big(x(t)-x^*\big)\big\Vert\right\}$ satisfies Eq. (\ref{eq:condition}). \\ We next show that within the interval given in Eq. (\ref{eq:condition}), the sequence $\left\{ \big\Vert D(t-1)^{1/2}\big(x(t)-x^*\big)\big\Vert\right\}$ decreases with a quadratic rate.}\\
For analysis simplicity we denote by $[\bar{t},\bar{t}+l]$ the interval in which Eq. (\ref{eq:condition}) is satisfied. Using the result of Lemma \ref{linquad} we have \begin{equation}\begin{aligned}&\mathbb{E}\left[\big\Vert D(\bar{t})^{1/2}\big(x(\bar{t}+1)-x^*\big)\big\Vert\right]\leq\\&\Gamma_1 \left(\mathbb{E}\left[\big\Vert D(\bar{t}-1)^{1/2}\big(x(\bar{t})-x^*\big)\big\Vert\right]\right)^2+\Gamma(t)\mathbb{E}\left[\big\Vert D(\bar{t}-1)^{1/2}\big(x(\bar{t})-x^*\big)\big\Vert\right].\label{eq:ineqE}\end{aligned}\end{equation}
\sm{We now use the left hand side of Eq. (\ref{eq:condition}) substitute the upper bound for $\Gamma(t)$ in Eq. (\ref{eq:ineqE}) and obtain}
\[\mathbb{E}\Big[\big\Vert D(\bar{t})^{1/2}\big(x(\bar{t}+1)-x^*\big)\big\Vert\Big]\leq\big(\Gamma_1+\frac{1}{\theta}\big)\Big(\mathbb{E}\Big[\big\Vert D(\bar{t}-1)^{1/2}\big(x(\bar{t})-x^*\big)\big\Vert\Big]\Big)^2.\]
By multiplying both sides of the previous inequality by $\frac{\theta\Gamma_1+1}{\theta}$, we have
\[\frac{\theta\Gamma_1+1}{\theta}\mathbb{E}\Big[\big\Vert D(\bar{t})^{1/2}\big(x(\bar{t}+1)-x^*\big)\big\Vert\Big]\leq\left(\frac{\theta\Gamma_1+1}{\theta}\mathbb{E}\Big[\big\Vert D(\bar{t}-1)^{1/2}\big(x(\bar{t})-x^*\big)\big\Vert\Big]\right)^2.\]
Applying this recursively up to any time $r\in[\bar{t},\bar{t}+l]$ \sm{and dividing both sides by $\frac{\theta\Gamma_1+1}{\theta}$} yields
\begin{equation}
\mathbb{E}\Big[\big\Vert D(r-1)^{1/2}\big(x(r)-x^*\big)\big\Vert\Big]\leq\frac{\theta}{\theta\Gamma_1+1}\Big(\frac{\theta\Gamma_1+1}{\theta}\mathbb{E}\Big[\big\Vert D(\bar{t}-1)^{1/2}\big(x(\bar{t})-x^*\big)\big\Vert\Big]\Big)^{2^{r-\bar{t}}}.\label{eq:quadx}
\end{equation}
We note that the right hand side of Eq. (\ref{eq:condition}) implies that 
$\frac{\theta\Gamma_1+1}{\theta}\mathbb{E}\Big[\big\Vert D(\bar{t}-1)^{1/2}\big(x(\bar{t})-x^*\big)\big\Vert\Big]<1$, hence Eq. (\ref{eq:quadx}) establishes the quadratic convergence rate for all $r\in[\bar{t},\bar{t}+l]$. 
\end{proof} 
\begin{figure}
  \centering
  \subfloat[Uniform activation]{\includegraphics[width=0.5\textwidth]{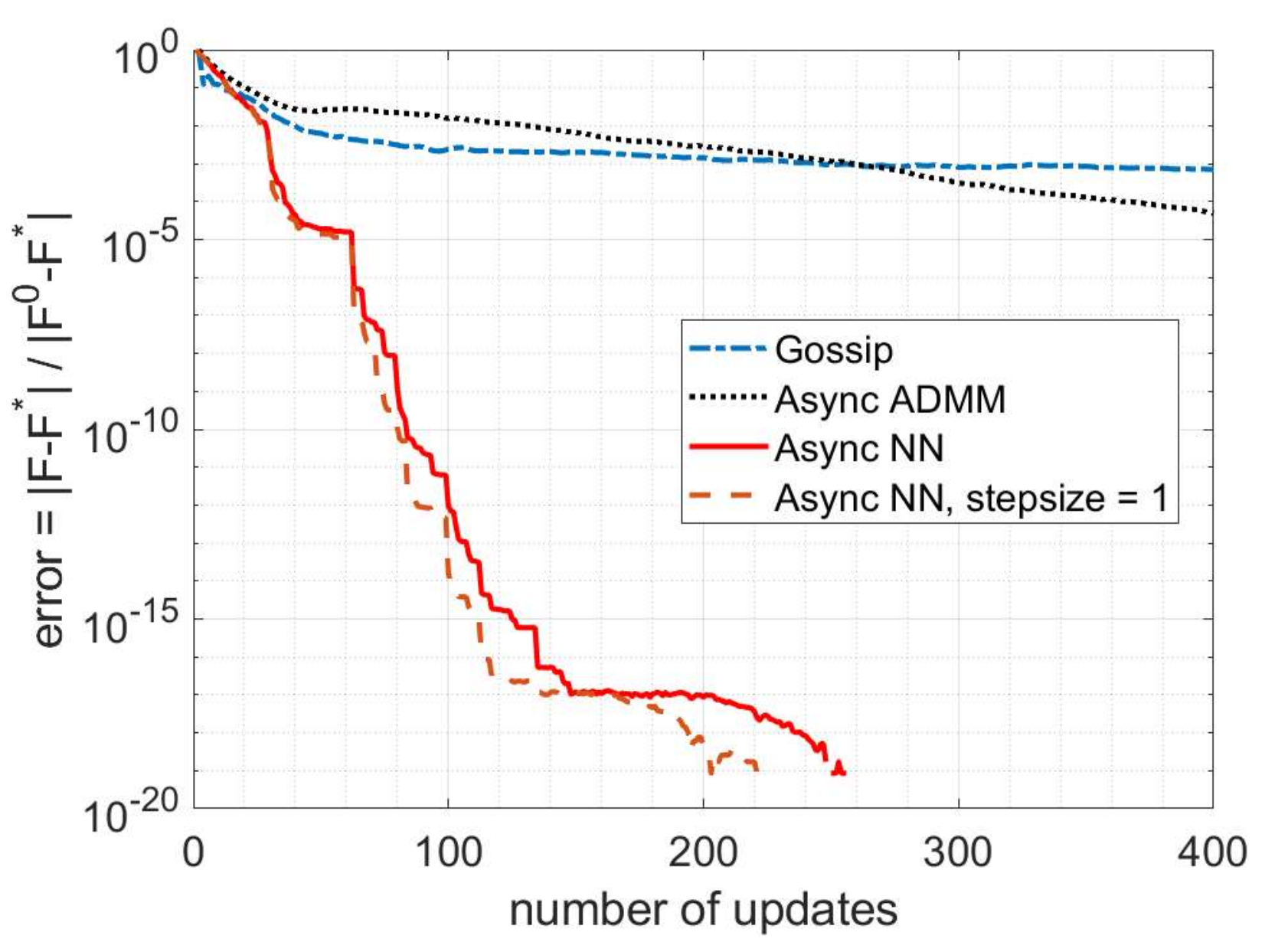}\label{fig:f1}}
  \hfill
  \subfloat[Nonuniform activation]{\includegraphics[width=0.5\textwidth]{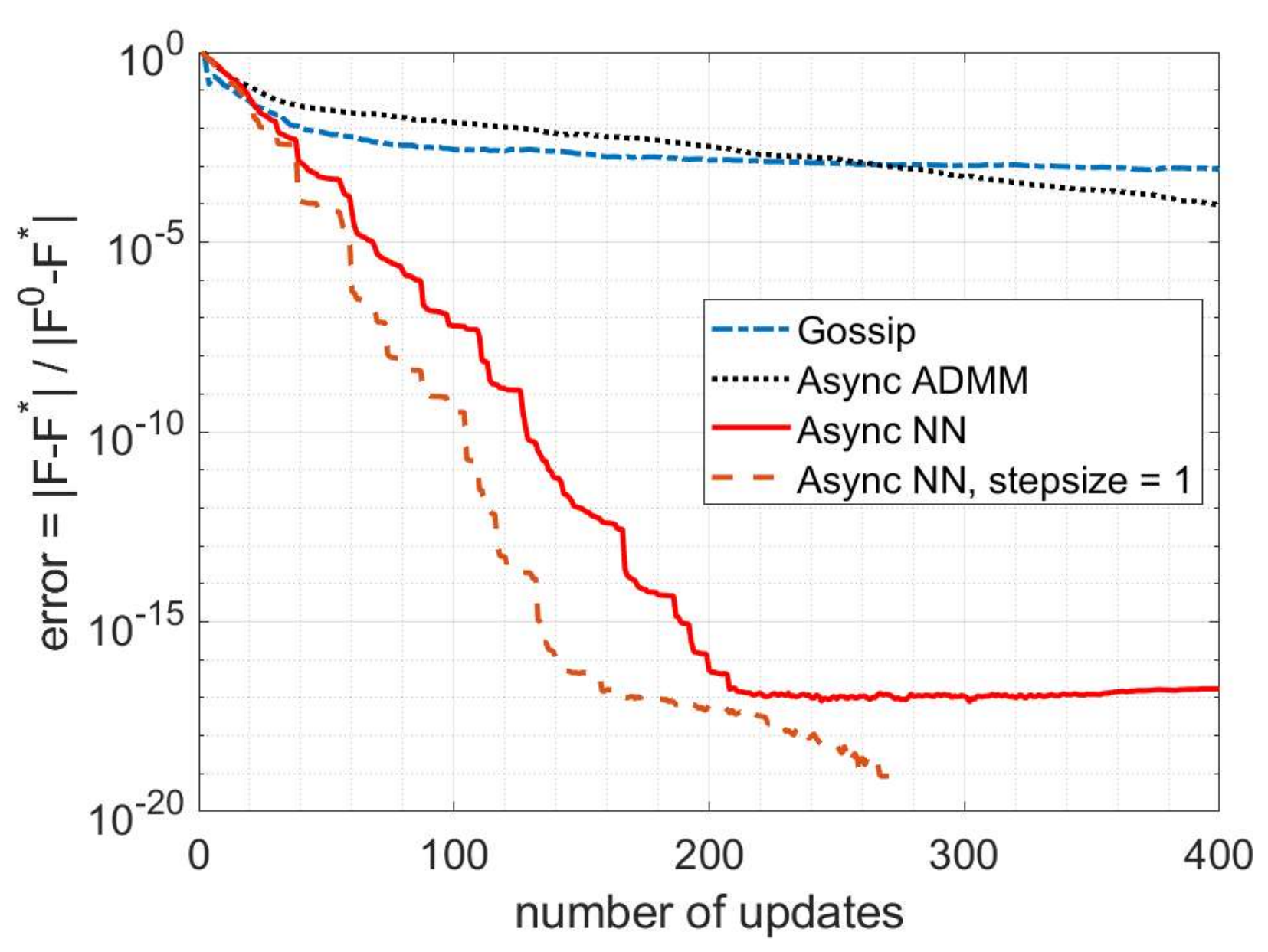}\label{fig:f2}}
  \caption{Convergence of asynchronous NN, asynchronous ADMM and gossip. Quadratic functions, complete graph.}\label{quadfig1}
\end{figure}
\re{\begin{remark}\label{localresult} According to Lemma \ref{linquad}, the expected value of the weighted error norm at each iteration, $\mathbb{E}\Big[\norm{D(t-1)^{1/2}\big(x(t)-x^*\big)}\Big]$, is upper bounded by terms that are quadratic and linear on the error associated with the previous iterate. Because of the linear term, the algorithm does not achieve the quadratic convergence to the solution as in Newton's method. However, as per Theorem \ref{localquad}, while the algorithm proceeds towards the solution, for an interval of iterations, in which the quadratic term dominates, the expected value of the weighted error norm decreases with a quadratic rate. We emphasize that in the synchronous network Newton algorithm, \cite{mok15}, the interval of quadratic convergence can be enlarged by using a better approximation of the Hessian inverse matrix, i.e., truncating the Taylor series [c.f. Eq. (\ref{eq:inverse})] with more terms,  which is associated with more communications. However, in the asynchronous algorithm, only the $0^{th}$ and $1^{st}$ terms can be used to approximate the Hessian inverse. Therefore, the length of the quadratic convergence interval only depends on function properties, network topology, and activation probabilities. Overall, our proposed method achieves linear, quadratic, and then linear rate of convergence and as it is proved in Theorem \ref{localquad} the quadratic convergence phase is not empty.
\end{remark}}
\section{Simulation Results}\label{sec:sim}
In this section, we present some numerical studies, where we compare the performance of the proposed asynchronous network Newton method \sf{with two existing totally asynchronous algorithms, asynchronous ADMM and asynchronous gossip, presented in \cite{we13} and \cite{ra10}.} It is important to note that gossip and asynchronous ADMM algorithms solve the constrained consensus problem, Eq. (\ref{consensusFormulation}), while the asynchronous network Newton algorithm solves the unconstrained penalized problem, Eq. (\ref{optFormulation}), for a fixed value of penalty constant $\alpha$. We note that the solutions of the two problems are different, resulting in different values of $F^*$. We also study the performance of our proposed algorithm on different networks. Finally, we compare the performance of asynchronous network Newton algorithm with its synchronous counterpart.   \par For all simulations we set the consensus matrix $W$ to be \re{$W=I-\frac{1}{d_{max}+1}L$, where $d_{max}$ is the largest element of the graph degree matrix $D$. The degree matrix of a graph is a diagonal matrix in which each diagonal entry is equal to the degree of the corresponding node, i.e., the total number of its neighbors. Matrix $L$ is the graph Laplacian matrix with $L=D-A$, where $A$ is the adjacency matrix with all the diagonal elements equal to zero and $A_{ij}=1$ if and only if node $i$ is connected to node $j$ and zero otherwise.} 

\subsection{Quadratic Objective Functions}
In this section, we present the simulation results for the case when the local objective functions are quadratic.
\re{We first consider a network of five agents which are connected through a complete graph, with the objective functions of the form $f_{i}(x_{i})=(x_{i}-i)^{2}$ , $i\in\{1,...,5\}$. For our asynchronous network Newton algorithm we choose the penalty parameter $\alpha=1$. We note that while the activation is uniform there is no need to scale the stepsize with the inverse of the probability matrix. In subfigure $(a)$ in Fig. \ref{quadfig1}, showing the results for uniform activation of the agents, we choose the stepsize $\varepsilon=0.9$ and in subfigure $(b)$ in Fig. \ref{quadfig1} and with nonuniform activation, we choose $\pi=\frac{2}{15}$ and $\varepsilon=0.12$ for our asynchronous network Newton method. In both uniform and nonuniform cases, the stepsize parameter is within the bounds given by Eq. (\ref{eq:step}). In both subfigures, for the gossip algorithm we use the diminishing stepsize of $\frac{1}{t}$ and for asynchronous ADMM we tune the stepsize to achieve the best performance. We run our simulation for $100$ different seeds and we plot the resulting average relative errors in the objective function value, $\frac{\vert F(x(t))-F^*\vert}{\vert F^0-F^*\vert}$. Asynchronous network Newton is the solid red line,  asynchronous gossip algorithm is the blue dot-dash line and asynchronous ADMM is the black dotted line. \sf{We also simulate the asynchronous network Newton algorithm, with $\frac{\varepsilon}{p_i}=1$ for all agents $i\in\{1, 2, ..., n\}$ [c.f. Algorithm \ref{async NN}, step 5]. This simulation is shown in the orange dash line.} We can see clearly that asynchronous network Newton outperforms the other two algorithms, which is expected due to the local quadratic rate. We have also simulated other objective function values and other network topologies and obtained similar results.}
\par \re{We next study the performance on our algorithm on  networks with different sizes and topologies and different quadratic objective functions. In subfigure (a) of Fig. \ref{compare}, we consider complete, cyclic (4-regular), path, ring, and random (connected Erdos-Renyi) graphs with $5-30$ agents. The objective function at each agent $i$ is of the form $f_i(x_i)=c_i(x_i-b_i)$, where $c_i$ and $b_i$ are integers, randomly chosen from $[1,100]$. For all the simulations in this subfigure, the penalty constant $\alpha=1$ and we choose the stepsize based on the bounds given by Eq. (\ref{eq:step}). We run the simulation for $100$ different seeds, with different objective functions, different activation patterns and different random graphs. We plot the average number of steps until the relative error is less than $\epsilon=0.01$, i.e., $\frac{\vert F(x(t))-F^*\vert}{\vert F^0-F^*\vert}<0.01$. We can see that, in all graph topologies, the number of steps, until reaching the $\epsilon-$neighborhood of the solution, increases with the number of agents in the network. This is expected due to the fact that in larger networks each agent is active less often and works with the information which is more outdated. We can also see that following the spectral gap properties, the complete graph results in smallest number of steps, while the path graph requires larger number of steps. However, in a network with complete graph, more communication is required in each step.}
 
\begin{figure}
  \centering
  \subfloat[Uniform activation]{\includegraphics[width=0.5\textwidth]{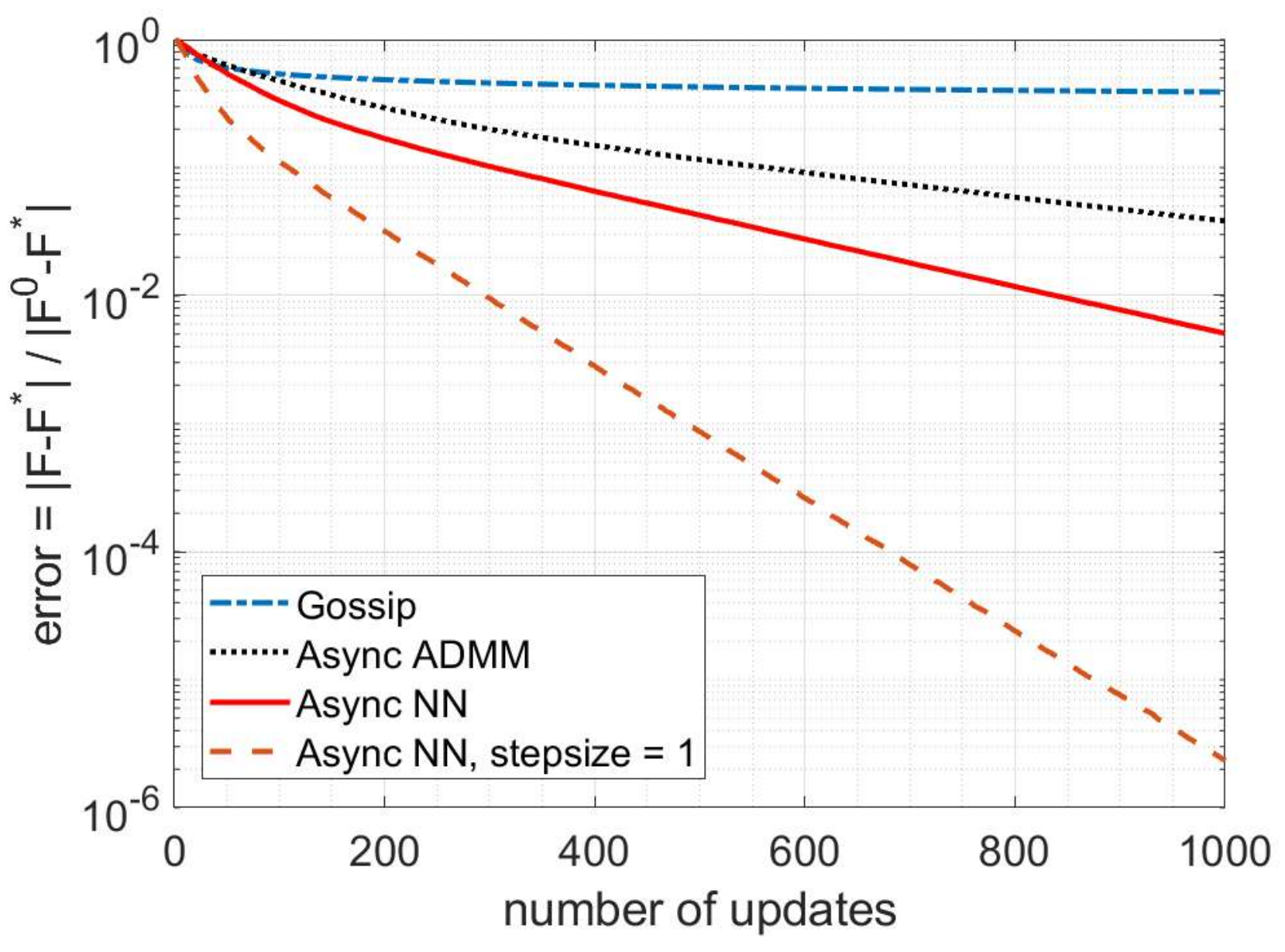}\label{fig:f1}}
  \hfill
  \subfloat[Nonuniform activation]{\includegraphics[width=0.5\textwidth]{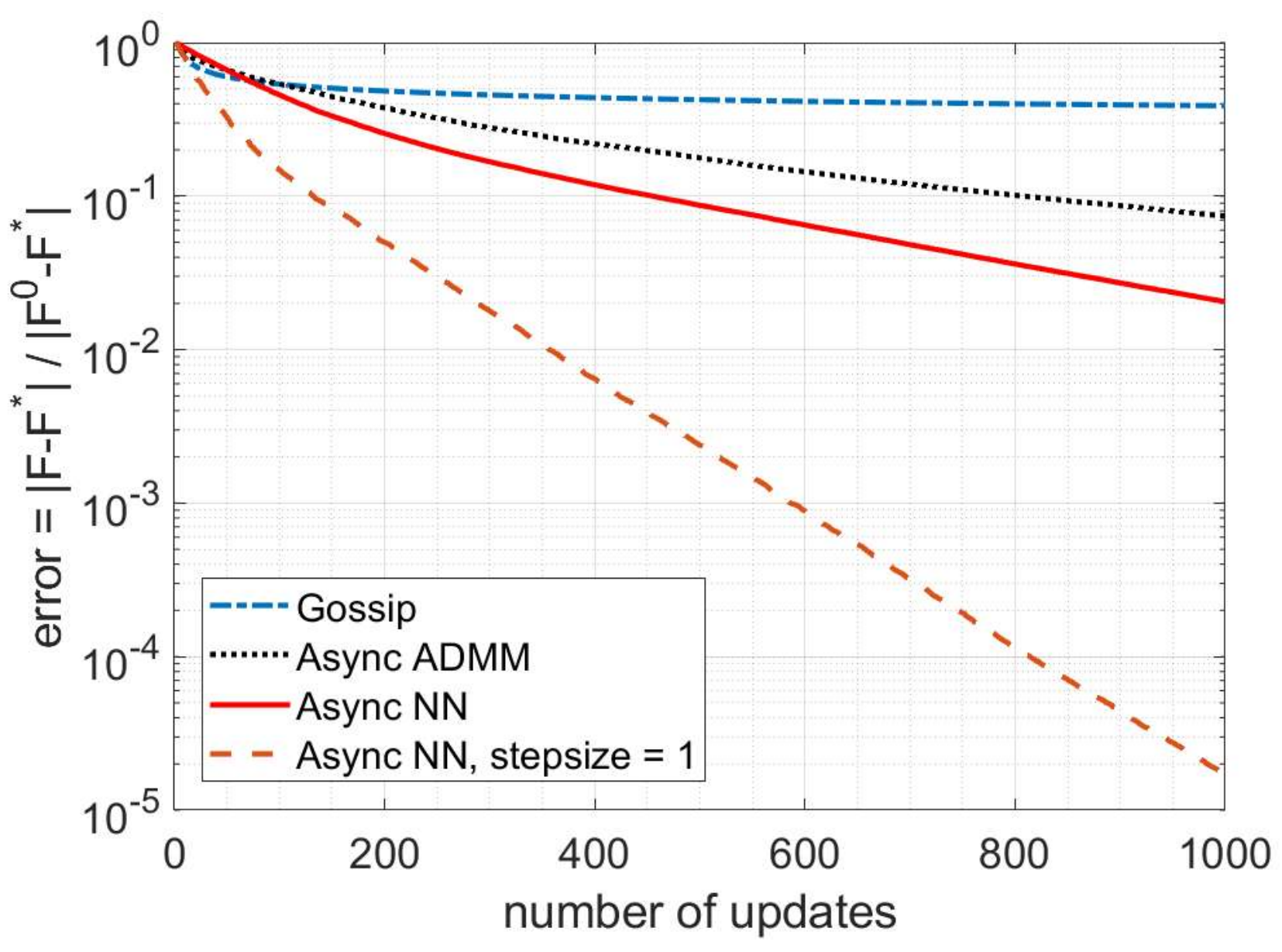}\label{fig:f2}}
  \caption{Convergence of asynchronous NN, asynchronous ADMM and gossip. Logistic regression, complete graph.}\label{logregfig1}
\end{figure}
\subsection{Non-quadratic Objective Functions} In order to study the performance of asynchronous network Newton algorithm for non-quadratic problems, we consider solving a classification problem using regularized logistic regression. We consider a problem with $K$ training samples that are uniformly distributed over $n=5$ agents in a network with complete graph. Each agent $i$ has access to $k_i=\floor{\frac{K}{n}}$ data points. This problem can be formulated as follows
\[\min_x f(x)=\frac{\upsilon}{2}\norm{x}^2+\frac{1}{K}\sum_{i=1}^n\sum_{j=1}^{k_i}\log\big[1+\exp(-v_{ij}u_{ij}x)\big],\]
where $u_{ij}$ and $v_{ij}$, $j\in\{ 1, 2, ..., k_i\}$ are the feature vector and the label for the data point $j$ associated with agent $i$ and the regularizer $\frac{\upsilon}{2}\big\Vert x\big\Vert^2$ is added to avoid overfitting. We can write this objective function in the form of $f(x)=\sum_{i=1}^n f_i(x)$, where $f_i(x)$ is defined as 
\[f_i(x)=\frac{\upsilon}{2n}\norm{x}^2+\frac{1}{K}\sum_{j=1}^{k_i}\log\big[1+\exp(-v_{ij}u_{ij}x)\big].\]
We are now able to define the local copies $x_i$ for each agent and form the penalized objective function [c.f. Eq. (\ref{eq:defF})].\\
\begin{figure}
  \centering
  \subfloat[Performance on different networks]{\includegraphics[width=0.49\textwidth]{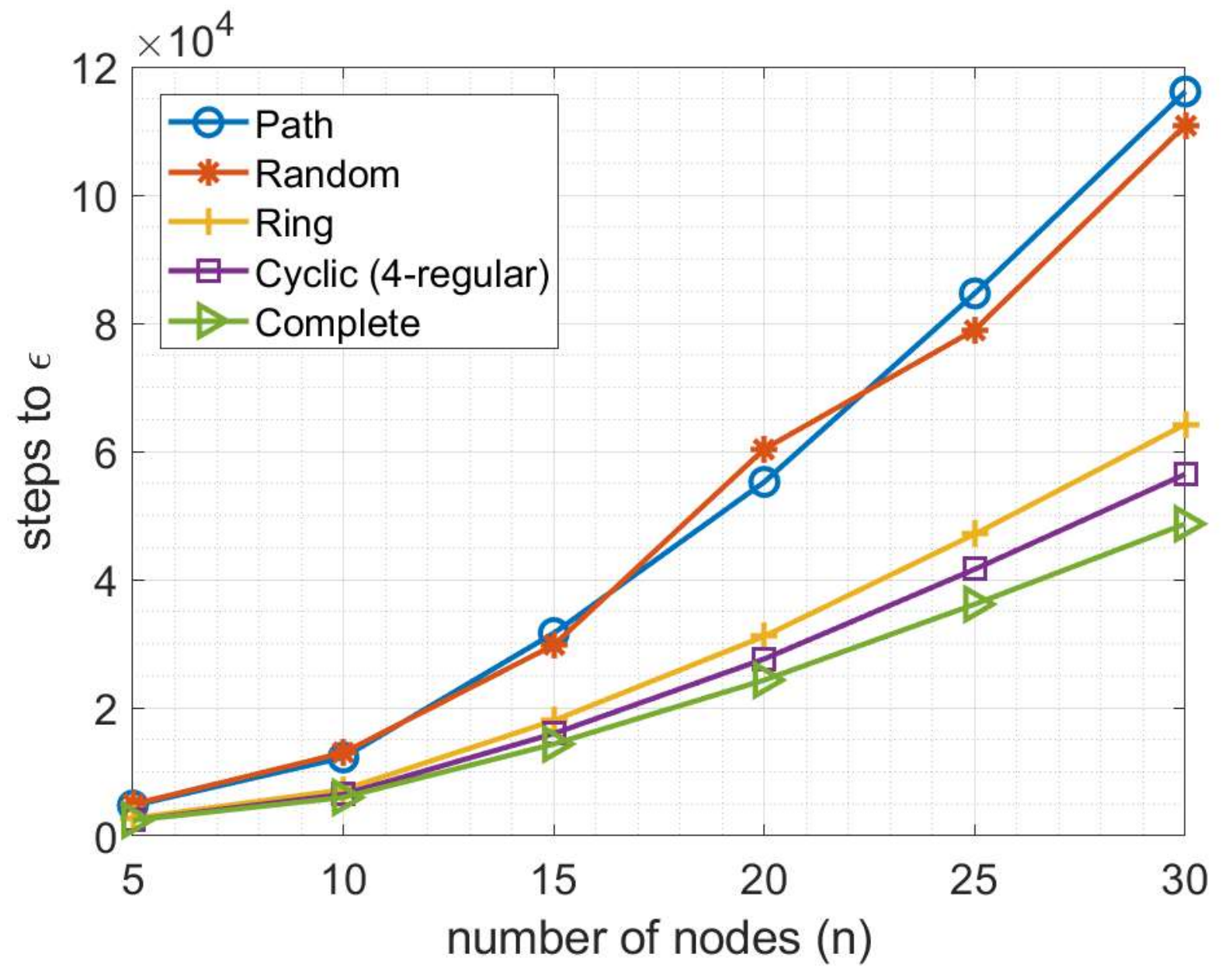}\label{fig:f1}}
  \hfill
  \subfloat[Comparison with synchronous NN ]{\includegraphics[width=0.51\textwidth]{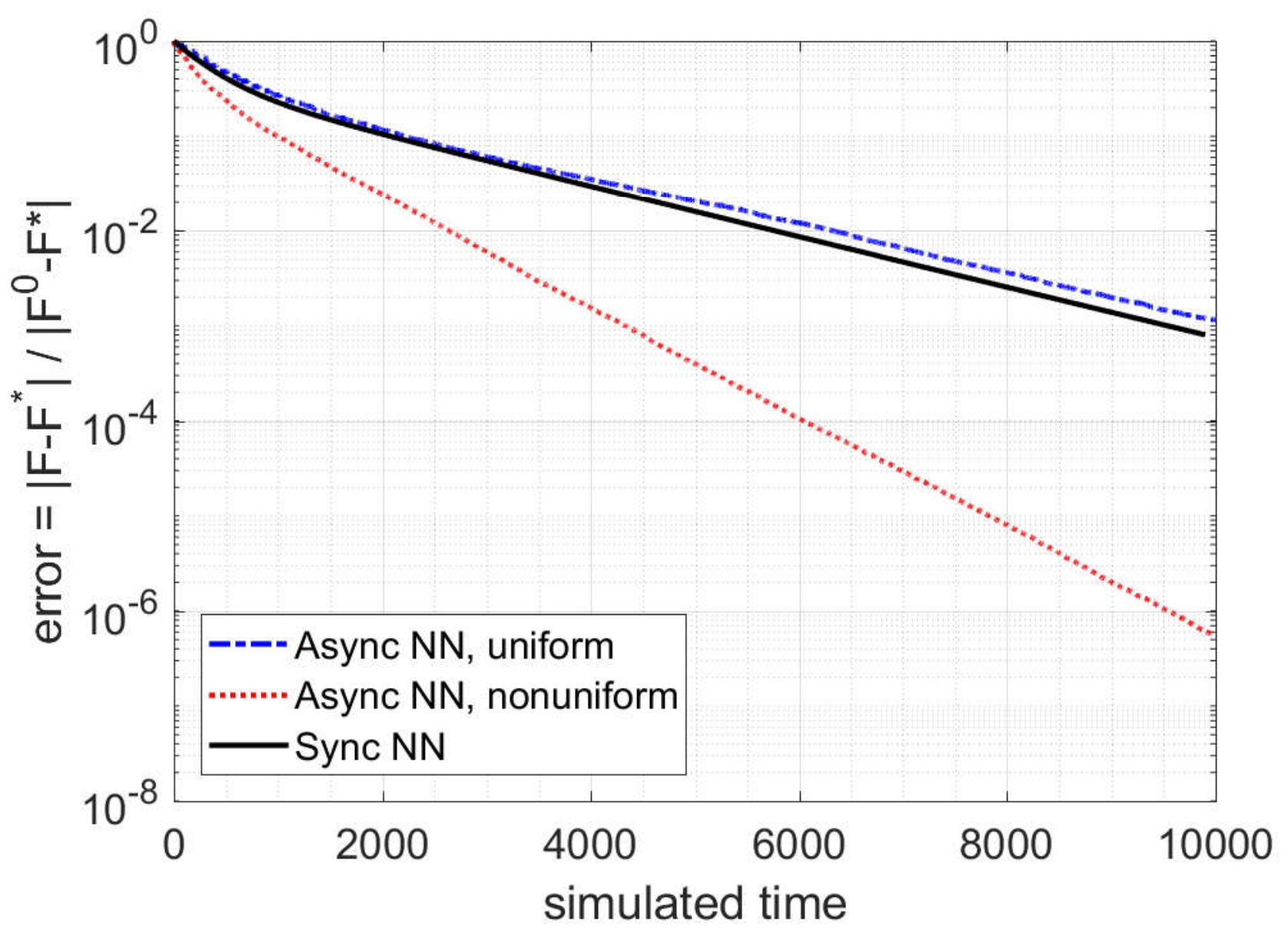}\label{fig:f2}}
  \caption{(a) Performance of asynchronous NN on different networks, quadratic cost functions. (b) Convergence of asynchronous NN and synchronous NN, Logistic regression, complete graph.}\label{compare}
\end{figure} \re{In our simulations, we use the diabetes-scale dataset \cite{chang2011libsvm},  with $768$ data points, each having a feature vector of size $8$ and a label which is either $1$ or $-1$. We distribute the data over five agents uniformly and study the performance of our algorithm on a network with complete graph. \sf{In both subfigures of Fig. \ref{logregfig1}, we use the diminishing stepsize of $\frac{1}{t}$ for gossip algorithm shown in the blue dot-dash line and for asynchronous ADMM, shown in the black dotted line, we tune the stepsize that gives the best performance. For asynchronous network Newton algorithm, shown in the red solid line, we consider the penalty coefficient of $\alpha=1$. In uniform activation case, subfigure $(a)$ in Fig. \ref{logregfig1}, we choose \re{$\varepsilon=0.35$} and for nonuniform activation case, subfigure $(b)$ in Fig. \ref{logregfig1}, we choose $\pi=\frac{2}{15}$ and \re{$\varepsilon=0.047$}, both of which are within the bounds given in Eq. (\ref{eq:step}). We also show the results with $\varepsilon=1$ for uniform activation and $\frac{\varepsilon}{p_i}=1$ [c.f. step 5 of Algorithm \ref{async NN}], for nonuniform activation in the orange dash line.} We run the simulation for $100$ different seeds and we plot the resulting average relative errors in the objective function value, $\frac{\vert F(x(t))-F^*\vert}{\vert F^0-F^*\vert}$.}
\par \re{Finally, we compare the performance of the asynchronous network Newton algorithm with its synchronous counterpart for the logistic regression problem for the same data set and network in subfigure (b) of Fig. \ref{compare}. We use the stepsize $\varepsilon=1$ for synchronous and asynchronous implementations. We run the asynchronous simulations for 100 different seeds and plot the average relative error. In the nonuniform activation case, we set the minimum activation probability to $\pi=\frac{1}{30}$. For the sake of comparison, we associate gradient evaluations with time units. We assume that one of the agents is $100$ times slower than the others.  In the synchronous algorithm, the agents need to wait for the slowest agent before proceeding to the next update, which results in slowdown at each iteration. In the asynchronous implementation,  this slowdown happens only when the slowest agent is active. The black solid line is the synchronous network Newton algorithm, the blue dash line is the asynchronous network Newton algorithm with uniform activation probabilities and the red dotted line is the asynchronous network Newton algorithm with nonuniform activation probabilities. We can see that the asynchronous network Newton algorithm with uniform activation probabilities is performing similar to its synchronous counterpart. However, if the slowest agent has a smaller activation probability, the asynchronous network Newton algorithm outperforms the synchronous implementation. Due to the nature of the Newton's method and the fact that the synchronous algorithm uses a better approximation of the Hessian inverse at each iteration, we do not expect the asynchronous algorithm to outperform the synchronous one, unless the slow agent is active less often.}
\par\re{We note that the horizontal axes of Fig. \ref{quadfig1} and Fig. \ref{logregfig1} represent the number of updates and not the iteration number. The reason for choosing this horizontal axis is to have a fair comparison, since in asynchronous ADMM and asynchronous gossip algorithms, two nodes are active and update at each iteration, while in  asynchronous network Newton algorithm one node updates its decision variable at each iteration. We also note that the running time for asynchronous ADMM algorithm is much longer than the other two algorithms, since it needs to solve a minimization problem per node activation. We notice that in gossip algorithm the active agent communicates with only one random neighbor while in our algorithm the active agent needs to communicate with all its neighbors. Therefore, each agent needs more storage if using the asynchronous network Newton algorithm. }
\section{Conclusion}\label{sec:con}
This paper presents an asynchronous distributed network Newton algorithm, in which the agents update randomly over time according to their local clocks. Such implementation removes the need for a central coordinator and enables the agents to work asynchronously from the others. We  show that the proposed method converges almost surely. \sf{We also establish global linear and local quadratic rate of convergence in expectation.} Simulation results show the convergence speed improvement of the asynchronous network Newton compared to the existing asynchronous ADMM and asynchronous gossip algorithms. Possible future work includes analysis of the convergence properties for a dynamic network and extending the convergence rate analysis to other second order asynchronous methods.
 
\bibliographystyle{plain}
\bibliography{citeNN}
\end{document}